\documentclass[a4paper,12pt]{amsart}

\usepackage[utf8]{inputenc}
\usepackage[table]{xcolor}
\usepackage{hyperref}
\usepackage{enumerate}
\usepackage{amsmath}
\usepackage{amssymb}
\usepackage{amsthm}
\usepackage{mathtools}
\usepackage{ytableau}
\usepackage[noabbrev,capitalize]{cleveref}
\usepackage{graphicx}
\usepackage[pdf]{graphviz}
\usepackage{caption, subcaption}
\usepackage{blkarray}
\usepackage{rotating}
\usepackage{calc}

\newtheorem{thm}{Theorem}[section]
\newtheorem{conj}[thm]{Conjecture}
\newtheorem{defn}[thm]{Definition}
\newtheorem{lem}[thm]{Lemma}
\newtheorem{prop}[thm]{Proposition}
\newtheorem{cor}[thm]{Corollary}
\newtheorem*{rmrk}{Remark}

\crefname{defn}{Definition}{Definitions}
\crefname{lem}{Lemma}{Lemmata}
\crefname{thm}{Theorem}{Theorems}
\crefname{prop}{Proposition}{Propositions}
\crefname{cor}{Corollary}{Corollaries}
\crefname{fig}{Figure}{Figures}
\crefname{conj}{Conjecture}{Conjectures}

\newcommand{\doi}[1]{\href{http://doi.org/#1}{\texttt{doi.org/#1}}}
\renewcommand{\S}{\mathfrak{S}}
\newcommand{\vertdots}{\substack{\boldsymbol{\cdot} \\ \boldsymbol{\cdot} \\ \boldsymbol{\cdot}}}

\begin{document}
\title{Properties of the Edelman--Greene bijection}
\author{Svante Linusson and Samu Potka}
\address{Department of Mathematics, KTH Royal Institute of Technology, Stockholm, Sweden.}
\email{\href{mailto:linusson@kth.se}{linusson@kth.se}, \href{mailto:potka@kth.se}{potka@kth.se}}
\thanks{The authors were supported by the Swedish Research Council, grant 621-2014-4780. An extended abstract of this paper appeared in the conference proceedings of FPSAC'18~\cite{LP18}.}

\begin{abstract}
Edelman and Greene constructed a correspondence between reduced words of the reverse permutation and standard Young tableaux. We prove that for any reduced word the shape of the region of the insertion tableau containing the smallest possible entries evolves exactly as the upper-left component of the permutation's (Rothe) diagram. Properties of the Edelman--Greene bijection restricted to 132-avoiding and 2143-avoiding permutations are presented. We also consider the Edelman--Greene bijection applied to non-reduced words.
\end{abstract}

\maketitle

\section{Introduction}
In 1982, Richard Stanley conjectured, and later proved algebraically in~\cite{St84} that the number of different reduced words for the reverse permutation in the symmetric group $\S_n$ is equal to the number of staircase shape standard Young tableaux. Motivated to find a bijective proof, Edelman and Greene~\cite{EG87} constructed a correspondence based on the celebrated Robinson--Schensted--Knuth (RSK) algorithm and Sch\"utzenberger's jeu de taquin. See also the work of Haiman on dual equivalence~\cite{Ha92}. Later, Little \cite{L03} found another bijection based on the Lascoux--Sch\"utzenberger tree,~\cite{LS85}, which was proved to be equivalent to the Edelman--Greene correspondence by Hamaker and Young in~\cite{HY14}. 

Recently, reduced words of the reverse permutation have been studied under the name of \emph{sorting networks}. Uniformly random sorting networks are the topic of, for example,~\cite{AHRV07} by Angel, Holroyd, Romik, and Vir\'ag, and the subsequent papers, in particular the recent work by Dauvergne and Vir\'ag~\cite{DV18} and Dauvergne~\cite{D18} announcing proofs of the conjectures in~\cite{AHRV07}. See an example of a sorting network illustrated by its \emph{wiring diagram} in \cref{wiring}.

\begin{figure}\label[fig]{wiring}
\includegraphics[width=0.95\textwidth]{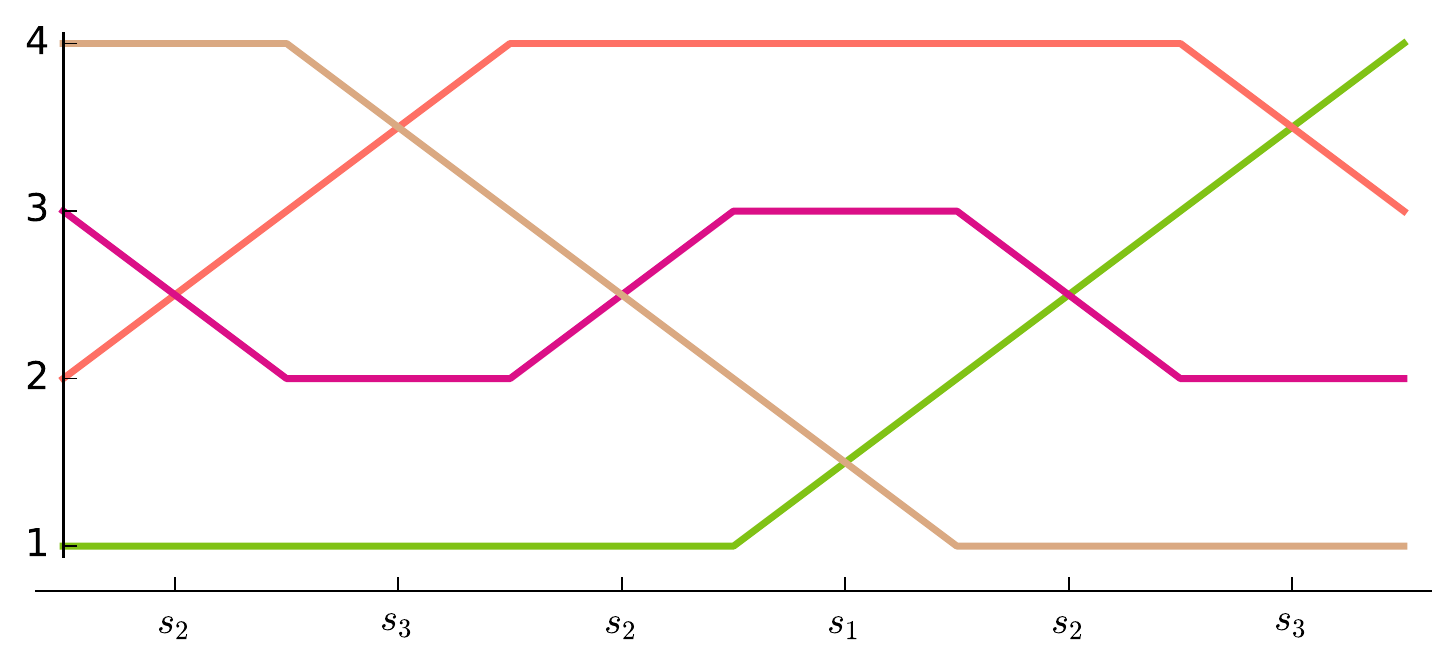}
\caption{The wiring diagram of 232123.}
\end{figure}

Our main result, \cref{frozendiagram}, is that the shape of the empty area (Rothe diagram) in the upper left corner of the permutation matrix is exactly the same as a region in the tableaux generated by the Edelman--Greene correspondence which we call the \emph{frozen region}. See \cref{diagram}. One consequence of this is a reformulation of a part of~\cite[Conjecture 2]{AHRV07} directly in terms of the Edelman--Greene bijection. See \cref{altconj} in Section 3.

As a byproduct of \cref{frozendiagram} we obtain some new observations and simple reproofs of previous results on the reduced words of 132-avoiding permutations in \cref{132av}, \cref{132f} and \cref{132sort}. We also consider sorting networks whose intermediate permutations are required to be 132-avoiding. These can be viewed as chains of maximum length in the Tamari lattice~\cite{BW97}, and have recently been studied by Fishel and Nelson~\cite{FN14}, and Schilling, Thi\'ery, White and Williams~\cite{STWW17}. The results in this paper are used to study limit phenomena of random 132-avoiding sorting networks in~\cite{LPS18}.

In \cref{S:nonreduced} we consider the Edelman--Greene bijection applied to non-reduced words. In particular, we study the sets of words yielding the same pairs of Young tableaux under the Edelman--Greene correspondence and study a natural partial order on this set which turns out to have some nice and surprising properties. Note that there is a different generalization of the Edelman--Greene bijection for non-reduced words called Hecke insertion~\cite{BKSTY08}.

\section{Preliminaries}
This section briefly reviews the basic definitions and background of this paper.
\subsection{Reduced words and the weak Bruhat order on \texorpdfstring{$\S_n$}{Sn}}
The \emph{symmetric group} $\S_n$ contains all permutations $\sigma = \sigma(1)\dots\sigma(n)$ on $[n] = \{1, \dots, n\}$.
The set of \emph{inversions} of a permutation $\sigma \in \S_n$ is defined as $\mathrm{Inv}(\sigma) = \{(i, j) : 1 \leq i < j \leq n, \sigma(i) > \sigma(j)\}$. 
The \emph{weak Bruhat order} is then defined by $\sigma \preceq_w \tau$ for $\sigma, \tau \in \S_n$ if $\mathrm{Inv}(\sigma) \subseteq \mathrm{Inv}(\tau)$. The \emph{reverse permutation} $n (n-1) \dots 2 1$ is the unique maximal element of $(\S_n, \preceq_w)$ and the identity permutation $\mathrm{id} = 1 \dots n$ the unique minimal element. 

A word $w = w_1 \dots w_m$ with letters $w_i \in [n - 1]$ defines a permutation $\sigma \in \S_n$ by $\sigma = s_{w_1} \dots s_{w_m}$, where $s_i$ is the adjacent transposition $(i \ i+1)$. The notation $w \in \mathbb{N}^*$ means that $w$ is a finite word with positive integer letters. We define len$(w) = m$, the length of $w$. When len$(w) = \mathrm{inv}(\sigma)$, we say that $w$ is a \emph{reduced word} of $\sigma$. Note that each reduced word $w = w_1 w_2 \dots w_m$, $m = \mathrm{inv}(\sigma)$, of $\sigma \in \S_n$ can be identified with a chain $\mathrm{id} \preceq_w  s_{w_1} \preceq_w s_{w_1}s_{w_2} \preceq_w \dots \preceq_w s_{w_1}s_{w_2} \cdots s_{w_m} = \sigma$ in the weak Bruhat order on $\S_n$. We denote the set of reduced words of $\sigma \in \S_n$ by $\mathcal{R}(\sigma)$, and, for convenience, in the case of $\sigma = n (n - 1) \dots 2 1$ use the abbreviation $\mathcal{R}(n)$.

We will adopt the convention that the permutation matrix corresponding to $\sigma\in \S_n$ has 1s in entries $(\sigma(i),i), i=1\dots n$, see the example below. This is the transpose of the perhaps more typical convention. It is important to note that we consider the transpositions acting on positions and perform the compositions of $s_{w_i}$ corresponding to a word $w = w_1 \dots w_m$ from the left in our arguments. (Equivalently one could compose from the right and consider them acting on values.) As an example, consider $\S_4$ and the reduced word $1213$. Composing $s_1 s_2 s_1 s_3$ from the left yields the permutation $3241$. In terms of permutation matrices, we would have
\[\begin{blockarray}{cccc}
\mathbf{2} & \mathbf{1} & \mathbf{3} & \mathbf{4} \\
\begin{block}{(cccc)}
0 & 1 & 0 & 0 \\
1 & 0 & 0 & 0 \\
0 & 0 & 1 & 0 \\
0 & 0 & 0 & 1 \\
\end{block}
\BAmulticolumn{4}{c}{s_1} \\
\end{blockarray}\quad \quad  \begin{blockarray}{cccc}
\mathbf{2} & \mathbf{3} & \mathbf{1} & \mathbf{4} \\
\begin{block}{(cccc)}
0 & 0 & 1 & 0 \\
1 & 0 & 0 & 0 \\
0 & 1 & 0 & 0 \\
0 & 0 & 0 & 1 \\
\end{block}
\BAmulticolumn{4}{c}{s_1 s_2} \\
\end{blockarray}\quad \quad  \begin{blockarray}{cccc}
\mathbf{3} & \mathbf{2} & \mathbf{1} & \mathbf{4} \\ 
\begin{block}{(cccc)}
0 & 0 & 1 & 0 \\
0 & 1 & 0 & 0 \\
1 & 0 & 0 & 0 \\
0 & 0 & 0 & 1 \\
\end{block}
\BAmulticolumn{4}{c}{s_1 s_2 s_1} \\
\end{blockarray} \quad \quad  \begin{blockarray}{cccc}
\mathbf{3} & \mathbf{2} & \mathbf{4} & \mathbf{1} \\ 
\begin{block}{(cccc)}
0 & 0 & 0 & 1 \\
0 & 1 & 0 & 0 \\
1 & 0 & 0 & 0 \\
0 & 0 & 1 & 0 \\
\end{block}
\BAmulticolumn{4}{c}{s_1 s_2 s_1 s_3} \\
\end{blockarray}\vspace{-\baselineskip}\] where we can see that $s_i$ corresponds to swapping the columns $i$ and $i+1$.

\subsection{Young tableaux}
Recall that a \emph{partition} $\lambda$ of $m \in \mathbb{N}$ is a tuple $(\lambda_1, \dots, \lambda_{\ell})$ of positive integers $\lambda_1 \geq \dots \geq \lambda_{\ell} > 0$ such that $\lambda_1 + \dots + \lambda_{\ell} = m$. The \emph{length} of $\lambda$ is the number of parts in it: $\mathrm{len}(\lambda) = \ell$. A partition can be represented by its \emph{Young diagram} (also called \emph{Ferrers diagram}) which is the set $\{(i, j) \in \mathbb{N}^2 : 1 \leq i \leq {\ell}, 1 \leq j \leq \lambda_i\}$ and is often (in the so-called English notation) drawn as a collection of square boxes corresponding to the \emph{cells} $(i, j)$ with $i$ increasing downwards and $j$ to the right.

A \emph{Young tableau} $T$ of shape $\lambda = (\lambda_1, \dots, \lambda_{\ell})$ is a filling of the Young diagram of $\lambda$, typically with positive integer \emph{entries}, denoted $T_{i, j}$. Such a tableau $T$ is called \emph{standard} if the entries $1, \dots, \lambda_1 + \dots + \lambda_{\ell}$ appear exactly once each, and the rows and columns of $T$ are strictly increasing. We let $\mathrm{SYT}(\lambda)$ be the set of standard Young tableaux of the shape $\lambda$.

The \emph{reading word} $r(T)$ of a Young tableau $T$ is the word obtained by collecting the entries of $T$ row by row from left to right starting from the bottom row.

\subsection{The Edelman--Greene bijection}
The Edelman--Greene correspondence is a bijection between $\mathcal{R}(n)$, that is, maximal chains in the weak Bruhat order on $\S_n$, and standard Young tableaux of the \emph{staircase shape} $\mathrm{sc}_n = (n-1, n-2, \dots, 1)$.
\begin{defn}[The Edelman--Greene insertion]
	Suppose that $P$ is a Young tableau with strictly increasing rows $P_1, \dots, P_{\ell}$ and $x_0 \in \mathbb{N}$ is to be inserted in $P$. The insertion procedure is as follows for each $0 \leq i \leq \ell$:\begin{itemize}
		\item If $x_i > z$ for all $z \in P_{i+1}$, place $x_i$ at the end of $P_{i+1}$ and stop.
		\item If $x_i = z'$ for some $z' \in P_{i+1}$, insert $x_{i+1} = z' + 1$ in $P_{i+2}$.
		\item Otherwise, $x_i < z$ for some $z \in P_{i+1}$, and we let $z'$ be the least such $z$, replace it by $x_i$ and insert $x_{i+1} = z'$ in $P_{i+2}$. In both this and the case above we say that $x_i$ \textbf{bumps} $z'$.
	\end{itemize} Repeat the insertion until for some $i$ the $x_i$ is inserted at the end of $P_{i+1}$ and the algorithm stops. This could be a previously empty row $P_{\ell + 1}$.
\end{defn}
We should mention that our definition of the insertion differs from that of \cite{EG87}, where it is called the \emph{Coxeter--Knuth insertion}. However, using for example the proof of \cite[Lemma 6.23]{EG87}, one can show that the two definitions coincide for reduced words. In our formulation the tableaux are increasing in rows and columns also for non-reduced words.
Note also that except for a difference in handling equal elements bumping, the Edelman--Greene insertion and the RSK insertion are the same.
\begin{defn}[The Edelman--Greene correspondence]\label[defn]{egbij}
	Suppose that $w = w_1 \dots w_i \dots w_m \in \mathbb{N}^*$. Initialize $P^{(0)} = \emptyset$.
	\begin{itemize}
		\item For each $1 \leq i \leq m$, insert $w_i$ in $P^{(i-1)}$ and denote the result by $P^{(i)}$.
	\end{itemize}
	Let $P^{(m)} = P(w)$ and let $Q(w)$ be the Young tableau obtained by setting $Q(w)_{i, j} = k$ for the unique cell $(i, j) \in P^{(k)}\setminus P^{(k-1)}$. Set $\mathrm{EG}(w) = Q(w)$.
\end{defn}

As an example, consider the reduced word $w = 321232$. Then the $P^{(k)}, 1 \leq k \leq 6$, form the following sequence
\[ \ytableausetup{centertableaux}
\begin{ytableau}
\scriptstyle 3 \\
\none \\
\none
\end{ytableau}\ \longrightarrow\ \begin{ytableau}
\scriptstyle 2 \\
\scriptstyle 3 \\
\none
\end{ytableau}\ \longrightarrow\ \begin{ytableau}
\scriptstyle 1 \\
\scriptstyle 2\\
\scriptstyle 3
\end{ytableau}\ \longrightarrow\ \begin{ytableau}
\scriptstyle 1 & \scriptstyle 2\\
\scriptstyle 2\\
\scriptstyle 3
\end{ytableau}\ \longrightarrow\ \begin{ytableau}
\scriptstyle 1 & \scriptstyle 2 & \scriptstyle 3\\
\scriptstyle 2\\
\scriptstyle 3
\end{ytableau}\ \longrightarrow\ \begin{ytableau}
\scriptstyle 1 & \scriptstyle 2 & \scriptstyle 3\\
\scriptstyle 2 & \scriptstyle 3 \\
\scriptstyle 3
\end{ytableau}\] so that \[ P(321232) = \begin{ytableau} \scriptstyle 1 & \scriptstyle 2 & \scriptstyle 3\\
\scriptstyle 2 & \scriptstyle 3 \\
\scriptstyle 3 \end{ytableau}\ \mathrm{and}\ \mathrm{EG}(321232) = Q(321232) = \begin{ytableau} \scriptstyle 1 & \scriptstyle 4 & \scriptstyle 5\\
\scriptstyle 2 & \scriptstyle 6 \\
\scriptstyle 3 \end{ytableau}.\]

The tableau $P(w)$ is called the \emph{insertion tableau} and the tableau $Q(w)$ the \emph{recording tableau}. Note that $P(w)$ and $Q(w)$ are always of the same shape for a fixed $w$. We now state one of the main results of Edelman and Greene.

\begin{thm}[{\cite[Theorem 6.25]{EG87}}]\label[thm]{egmain}
The correspondence \[w \mapsto (P(w), Q(w))\] is a bijection between $\cup_{\sigma \in \S_n} \mathcal{R}(\sigma)$ and the set of pairs of tableaux $(P, Q)$ such that $P$ is row and column strict, $r(P)$ is reduced, $P$ and $Q$ have the same shape, and $Q$ is standard.
\end{thm}

Each of the $P^{(k)}$, $1 \leq k \leq m$, is going to contain some number of entries such that $P^{(k)}_{i, j} = i + j - 1$. We call the region of $P^{(k)}$ formed by such entries the \emph{frozen region} and say that an insertion tableau is \emph{frozen} if the tableau is entirely a frozen region. The reason for using this terminology is that the frozen region does not change during the Edelman--Greene insertion. See $P$ in \cref{diagram}. The frozen region is white in the example. It turns out that $P(w)$ is always frozen when $w \in \mathcal{R}(n)$, and in fact, as we will see later in \cref{frozenp}, more generally if and only if $w \in \mathcal{R}(\sigma)$ with $\sigma$ 132-avoiding. Frozen tableaux have previously appeared in the literature on the combinatorics of K-theory under the name \emph{minimal increasing tableaux}, see, for example,~\cite{BS16} and~\cite{HKPWZZ17}.

\begin{thm}[{\cite[Theorem 6.26]{EG87}}]\label[thm]{egstaircase}
	Suppose $w \in \mathcal{R}(n)$. Then $P(w)$ is frozen and $Q(w) \in \mathrm{SYT}(\mathrm{sc}_n)$. The map $\mathrm{EG}(w): w \mapsto Q(w)$ is a bijection from $\mathcal{R}(n)$ to $\mathrm{SYT}(\mathrm{sc}_n)$.
\end{thm}

Continuing in the setting of \cref{egstaircase}, if $w \in \mathcal{R}(n)$, the inverse to the Edelman--Greene bijection takes a very special form. To define it, we have to introduce Sch\"utzenberger's \emph{jeu de taquin}. For a good introduction, we refer to~\cite{St99} or~\cite{Sa01}, although the terminology is slightly different. 

Let $T$ be a partially filled Young diagram with increasing rows and columns, and each entry $1 \leq k \leq \max_{(i, j) \in T} T_{i, j}$ occurring exactly once. The \emph{evacuation path} of $T$ is a sequence of cells $\pi_1, \dots, \pi_s$ such that \begin{itemize}
	\item $\pi_1 = (i_{\mathrm{max}}, j_{\mathrm{max}})$, the location of the largest entry of $T$,
	\item if $\pi_k = (i, j)$, then $\pi_{k+1} = (i', j') \in T$ such that $T_{i', j'} = \max \{T_{i, j-1}, T_{i-1, j}\} > -\infty$ with the convention $T_{i, j} = -\infty$ for $(i, j) \not \in T$ and for unlabeled $(i, j) \in T$.
\end{itemize} 
Next, define the tableau $T^{\partial}$ by \begin{itemize}
	\item removing the label of $T_{\pi_1}$,
	\item and sliding the labels along the evacuation path: $T_{\pi_1} \leftarrow T_{\pi_2} \leftarrow \dots \leftarrow T_{\pi_s}$.
\end{itemize}
A single application of $\partial$ is called an \emph{elementary promotion}. Whenever a label $1 \leq \ell \leq T_{\pi_1}$ slides from some cell $(i, j)$ to $(i, j+1)$ (respectively $(i+1, j)$) in applying $\partial$ until all labels have been removed is referred to as a \emph{right slide} (respectively \emph{downslide}). For $w \in \mathcal{R}(n)$, the inverse to the Edelman--Greene bijection can then be defined as follows. 

\begin{thm}[{\cite[Theorem 7.18]{EG87}}]
	Suppose $Q \in \mathrm{SYT}(\mathrm{sc}_n)$. Apply $\partial$ until all labels have been cleared and say that $\pi_1^{(k)} = (i_k, j_k)$ is the first cell of the evacuation path $\pi^{(k)}$ for the $k$:th iteration. Then $\mathrm{EG}^{-1}(Q) = j_{\binom{n}{2}} \dots j_k \dots j_1$.
\end{thm}

Consider again the example following \cref{egbij}. Applying $\partial$ yields the sequence \begin{align*}
Q = \begin{ytableau} \scriptstyle 1 & \scriptstyle 4 & \scriptstyle 5\\
\scriptstyle 2 & \scriptstyle 6 \\
\scriptstyle 3 \end{ytableau}\ &\overset{\partial}{\longrightarrow}\ \begin{ytableau} \  & \scriptstyle 1 & \scriptstyle 5\\
\scriptstyle 2 & \scriptstyle 4 \\
\scriptstyle 3 \end{ytableau}\ \overset{\partial}{\longrightarrow}\ \begin{ytableau} \ & \ & \scriptstyle 1\\
\scriptstyle 2 & \scriptstyle 4 \\
\scriptstyle 3 \end{ytableau}\ \overset{\partial}{\longrightarrow}\ \begin{ytableau} \ & \ & \scriptstyle 1\\
\ & \scriptstyle 2 \\
\scriptstyle 3 \end{ytableau}\ \\ &\overset{\partial}{\longrightarrow}\ \begin{ytableau} \ & \ & \scriptstyle 1\\
\ & \scriptstyle 2 \\ \ \end{ytableau}\ \overset{\partial}{\longrightarrow}\ \begin{ytableau} \ & \ & \scriptstyle 1 \\ \ & \ \\ \ \end{ytableau}\ \overset{\partial}{\longrightarrow}\ \begin{ytableau} \ & \ & \ \\ \ & \ \\ \ \end{ytableau}.\end{align*} The largest entries are in the cells $\pi_1^{(1)} = (2, 2), \pi_1^{(2)} = (1, 3) , \pi_1^{(3)} = (2, 2), \pi_1^{(4)} = (3, 1), \pi_1^{(5)} = (2, 2)$ and $\pi_1^{(6)} = (1, 3)$. Hence, $\mathrm{EG}^{-1}(Q) = 321232$ as expected.

Another important operator will be the so-called \emph{evacuation} $S$, which is in some sense dual to promotion. If $T$ is a standard Young tableau, $T^S$ is defined by setting $T^S_{i, j} = k$ if and only if $(i, j)$ is not labeled in $T^{\partial^{k}}$ but is labeled in $T^{\partial^{k-1}}$. Thus $T^S$ records when cells become empty in iterating the elementary promotion $\partial$ for $T$. Returning to the previous example, we would have \[
Q^S = \begin{ytableau} \scriptstyle 1 & \scriptstyle 2 & \scriptstyle 6\\
\scriptstyle 3 & \scriptstyle 5 \\
\scriptstyle 4 \end{ytableau}.\] In his original work~\cite{Schu63}, Sch\"utzenberger proved a remarkable property of the operator $S$: it is an involution.

\section{Frozen regions and diagrams}
This section aims to prove our main result. Before proceeding with the proof, we need to review some additional properties of the Edelman--Greene bijection. The results below are due to Edelman and Greene. Recall that the reading word $r(P)$ of a Young tableau $P$ is formed by reading the entries of $P$ row by row, starting from the bottom row.
\begin{lem}[{\cite[Lemma 6.22]{EG87}}]\label[lem]{reading}
If $P$ is row and column strict, then $P(r(P)) = P$. 
\end{lem}
\begin{lem}[{\cite[a part of Lemma 6.23]{EG87}}]\label[lem]{cox}
If $w \in \mathcal{R}(\sigma)$, then $P(w)$ is row and column strict, and $r(P(w)) \in \mathcal{R}(\sigma)$.
\end{lem}
Our goal is to show that the shape of the frozen region of $P^{(k)}$ corresponds to the shape of one part of the so-called diagram of $\sigma = s_{w_1}s_{w_2} \dots s_{w_k}$. The (Rothe) \emph{diagram} $D(\sigma)$ of a permutation $\sigma$ is the set of cells left unshaded when we shade all the cells weakly to the east and south of 1-entries in the permutation matrix $M(\sigma)$. In particular, we consider the (possibly empty) connected component of $D(\sigma)$ containing $(1, 1)$ which we call the \emph{top-left component} of the diagram and denote by $D_{(1,1)}(\sigma)$. The top-left component induces a partition which is denoted by $\lambda(\sigma)$. Similarly, the frozen region of the insertion tableau of a reduced word induces a partition $\lambda_f(w)$ since by \cref{egmain} the tableau is row and column strict.
See \cref{diagram} for an example.

\begin{figure}[htbp!]
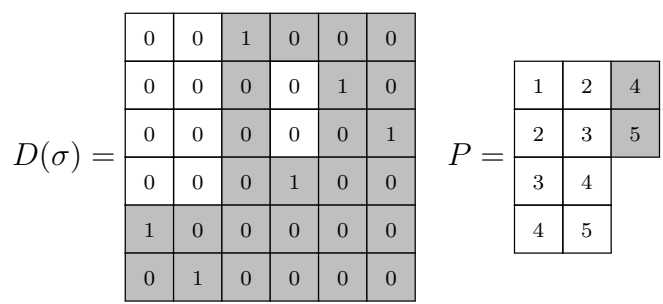

	\centering $D(\sigma) =
	\begin{ytableau}
		\scriptstyle 0 & \scriptstyle 0 & *(lightgray)\scriptstyle 1 & *(lightgray)\scriptstyle 0 & *(lightgray)\scriptstyle 0 & *(lightgray)\scriptstyle 0 \\
		\scriptstyle 0 & \scriptstyle 0 & *(lightgray)\scriptstyle 0 & \scriptstyle 0 & *(lightgray) \scriptstyle 1 & *(lightgray)\scriptstyle 0 \\
		\scriptstyle 0 & \scriptstyle 0 & *(lightgray)\scriptstyle 0 & \scriptstyle 0 & *(lightgray)\scriptstyle 0 & *(lightgray)\scriptstyle 1 \\
		\scriptstyle 0 & \scriptstyle 0 & *(lightgray)\scriptstyle 0 & *(lightgray)\scriptstyle 1 & *(lightgray)\scriptstyle 0 & *(lightgray)\scriptstyle 0 \\
		*(lightgray)\scriptstyle 1 & *(lightgray)\scriptstyle 0 & *(lightgray)\scriptstyle 0 & *(lightgray)\scriptstyle 0 & *(lightgray)\scriptstyle 0 & *(lightgray)\scriptstyle 0 \\
		*(lightgray)\scriptstyle 0 & *(lightgray)\scriptstyle 1 & *(lightgray)\scriptstyle 0 & *(lightgray)\scriptstyle 0 & *(lightgray)\scriptstyle 0 & *(lightgray)\scriptstyle 0 \\
	\end{ytableau}\quad
	P = \begin{ytableau}
	\scriptstyle 1 & \scriptstyle 2 & *(lightgray) \scriptstyle 4\\
	\scriptstyle 2 & \scriptstyle 3 & *(lightgray) \scriptstyle 5\\
	\scriptstyle 3 & \scriptstyle 4\\
	\scriptstyle 4 & \scriptstyle 5\\
	\end{ytableau}$
	\caption{The diagram $D(\sigma)$ and $P = P(w)$ for any $w \in \mathcal{R}(\sigma)$ for $\sigma = 561423$. The top-left component $D_{(1,1)}(\sigma)$ induces the partition $\lambda(\sigma) = (2, 2, 2, 2)$ and the frozen region of $P$ the partition $\lambda_f(w) = (2, 2, 2, 2)$.}
	\label[fig]{diagram}
\end{figure}
The following is one of our main results.
\begin{thm}\label[thm]{frozendiagram}
	If $w = w_1 \cdots w_{\ell}$ is reduced, then $\lambda(s_{w_1}\dots s_{w_{\ell}}) = \lambda_f(w)$. That is, the top-left component of the diagram of $s_{w_1}\dots s_{w_{\ell}}$ has the same shape as the frozen region of $P(w)$.
\end{thm}
\begin{proof} By \cref{reading} and \cref{cox}, it is enough to consider the case when $w = r(P(w)) = p_m \dots p_2 p_1$ where $p_i$ is the word formed by the letters in $P_i(w)$, the $i$:th row of $P(w)$. The remark below will be useful throughout.
\begin{rmrk}
	Let $\sigma(w) = s_{w_1} \cdots s_{w_k}$ for a word $w = w_1 \dots w_k$. Since $p_i, 1 \leq i \leq m$, is a strictly increasing word, a number in $\sigma(w)$ can move at most one step to the left in $\sigma(w)\sigma(p_i)$. The number in position $j$ moves $k$ steps to the right if $j(j+1)\cdots(j+k-1)$ is a subword of $p_i$.
\end{rmrk}

Let $\sigma = s_{w_1} \cdots s_{w_{\ell}}$. We will start by showing $\mathrm{len}(\lambda_f(w)) = \mathrm{len}(\lambda(\sigma))$, that is, there is a 1 in $(\mathrm{len}(\lambda_f(w)) + 1, 1)$ in $M(\sigma)$. By column and row-strictness, no transpositions in $p_m \dots p_i, i = \mathrm{len}(\lambda_f(w)) + 1$, affect the number $\mathrm{len}(\lambda_f(w)) + 1$ in the permutation. Hence, the first numbers $\mathrm{len}(\lambda_f(w)), \dots, 1$ of the rows $\mathrm{len}(\lambda_f(w)), \dots, 1$ form a (non-consecutive) sequence of transpositions $s_{\mathrm{len}(\lambda_f(w))}, \dots, s_1$ moving the number $\mathrm{len}(\lambda_f(w)) + 1$ to position 1 in $\sigma$. Thus $\mathrm{len}(\lambda_f(w)) = \mathrm{len}(\lambda(\sigma))$.

It remains to show that $\lambda_f(w)_i = \lambda(\sigma)_i$ for $1 \leq i \leq \mathrm{len}(\lambda_f(w))$. Consider the row $i$ with frozen part of length $\lambda_f(w)_i$, and the permutation $\sigma^i$ corresponding to the reduced word $w^i = p_m \cdots p_i$. By row and column-strictness, the letter $i$ does not appear in $w^{i+1}$. By the remark, the effect of the transpositions at indices $(p_i)_1 = i, \dots, (p_i)_{\lambda_f(w)_i} = (i + \lambda_f(w)_i - 1)$ is to move the 1 in row $i$ to column $i + \lambda_f(w)_i$ in $M(\sigma^i)$. Suppose $\lambda_f(w)_j = \lambda_f(w)_i$ for $i' \leq j \leq i$, and $\lambda_f(w)_j > \lambda_f(w)_i$ for all $1 \leq j < i'$. We will now prove that $\lambda(\sigma)_j = \lambda_f(w)_i$ for $i' \leq j \leq i$. This situation is illustrated in \cref{proofdiag}.\begin{figure}[htbp!]
\centering
\[
\begin{array}{c||c|c|c|c|c}
 & 1 & & \lambda_f(w)_i & \lambda_f(w)_i + 1 & \\	
\hline\hline
 & \vdots & \vdots & \vdots & \vdots & \\\hline
i' & 0 & \dots & 0 & \cellcolor{lightgray}1 & \cellcolor{lightgray} \dots\\\hline
& \vdots & \ddots & \vdots & \cellcolor{lightgray} \vdots &  \cellcolor{lightgray} \dots\\\hline
i & 0 & \dots & 0 & \cellcolor{lightgray}0 & \cellcolor{lightgray} \dots\\\hline
& \vdots & \cellcolor{lightgray}\vdots & \cellcolor{lightgray}\vdots & \cellcolor{lightgray}\vdots & \cellcolor{lightgray} \\

\end{array}
\]
\caption{The top-left component of the diagram of $\sigma$.}
\label[fig]{proofdiag}
\end{figure}

First, note that by the remark, in fact, for all $i' \leq j \leq i$, the 1 in row $j$ is moved by $p_j$ to column $j + \lambda_f(w)_i$ in $M(\sigma^j)$.

Next, consider $i'$. We have $P_{(j, \lambda_f(w)_{i'} + 1)} = j + \lambda_f(w)_{i}$ for $j < i'$ since it is in the frozen part. These entries for $j = i' - 1, \dots, 1$ will move the number $i'$ back to $\lambda_f(w)_i + 1 $. Hence $\sigma(i') = \lambda_f(w)_i + 1$. 

Finally, by the remark, the number $j$, $i' < j \leq i$, can be moved at most $j - 1$ steps to the left by $p_{j - 1}, \dots, p_1$. Hence $\sigma(j) > \lambda_f(w)_i + 1$ for $i' < j \leq i$, and the claim follows. This implies that $\lambda_f(w)_i = \lambda(\sigma)_i$ for any $1 \leq i \leq \mathrm{len}(\lambda_f(w))$.
\end{proof}

Zachary Hamaker has pointed out that the connection between frozen regions and permutation diagrams can also be understood using the theory on Hecke insertion.

Given $w = w_1 \dots w_k \in \mathcal{R}(\sigma)$, let $w^{\mathrm{rev}} = w_k \dots w_1 \in \mathcal{R}(\sigma^{-1})$. This corresponds to reflecting the wiring diagram in the vertical axis through the midpoint. Edelman and Greene proved the following lemmas.
\begin{lem}[{\cite[Corollary 7.22]{EG87}}]\label[lem]{reverse}
	Suppose $w = w_1 \dots w_k$ is a reduced word. Then
	\begin{itemize}
		\item $P(w^{\mathrm{rev}}) = P(w)^t$, where $t$ is the transpose, and
		\item $Q(w^{\mathrm{rev}}) = Q(w)^S$.
	\end{itemize}
\end{lem}
A similar statement holds for taking complements. This time the wiring diagram picture would be to reflect the diagram in the horizontal axis through the middle.
\begin{lem}[{\cite[Corollary 7.21]{EG87}}]\label[lem]{complement}
	Suppose $w = w_1 \dots w_k \in \mathcal{R}(n)$ and let $\bar{w} = \bar{w}_1 \dots \bar{w}_k$, where $\bar{w}_i = n - w_i$ for $1 \leq i \leq k$. Then $\bar{w} \in \mathcal{R}(n)$, and $Q(\bar{w}) = Q(w)^t$.
\end{lem}
Note that if $s_{w_1} \cdots s_{w_k} = \sigma$, then $s_{\bar{w}_1} \cdots s_{\bar{w}_k} = (\sigma^r)^c = (\sigma^c)^r$, where $\sigma^r = \sigma(n) \dots \sigma(1)$, which corresponds to flipping the permutation matrix about its vertical axis, and $\sigma^c = n + 1 - \sigma(1) \dots n + 1 - \sigma(n)$, which corresponds to doing the same about the horizontal axis.

The symmetries above yield the reformulation of a part of \cite[Conjecture 2]{AHRV07} below. We state it informally. The reader is referred to \cite{AHRV07} for the details on their conjecture.
\begin{conj}[Reformulation of a consequence of {\cite[Conjecture 2]{AHRV07}}]\label[conj]{altconj}
Let $w$ be a random sorting network. For a fixed $t \in (0, 1)$, the boundary of the limit shape $(n \rightarrow \infty)$ of the scaled frozen region \[F_t = \{ (\frac{2j}{n}-1, 1-\frac{2i}{n}) \in \mathbb{R}^2: (i, j) \in \lambda_f(w_1\dots w_{\lfloor t \binom{n}{2}\rfloor})\}\] is $\{(x, y) \in \mathbb{R}^2 : x \leq -\cos(\pi t), y \geq \cos(\pi t), \sin^2(\pi t) - 2xy\cos(\pi t) - x^2 - y^2 = 0\}.$
\end{conj}
A proof of the corresponding part of \cite[Conjecture 2]{AHRV07} has been announced recently in~\cite{DV18}. See also a stronger version in~\cite[Theorem 2]{D18}. \cref{altconj} and \cite[Conjecture 2]{AHRV07} are illustrated in \cref{comparison}.

\begin{figure}[htbp!]
\centering
\includegraphics[scale=0.42]{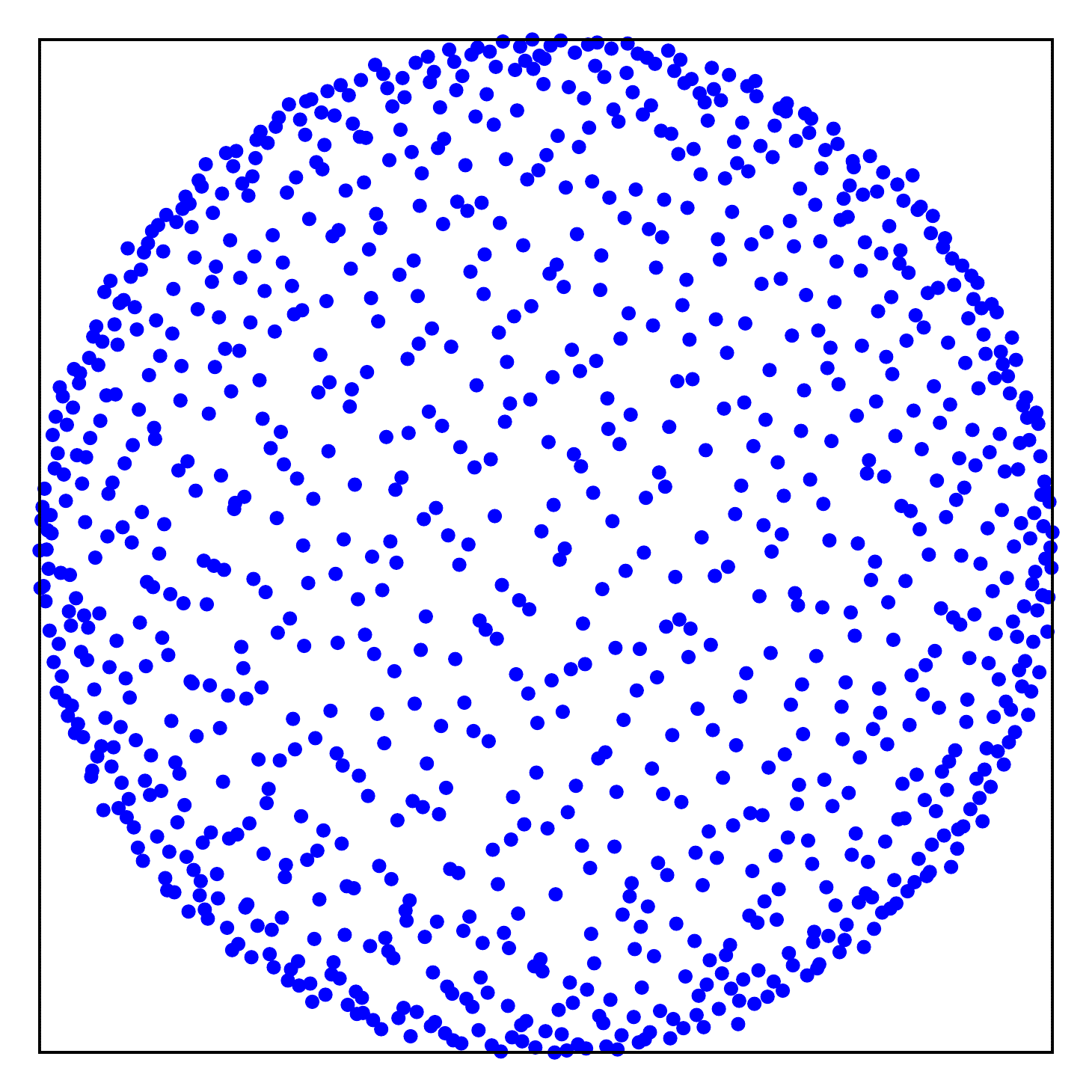}
\includegraphics[scale=0.42]{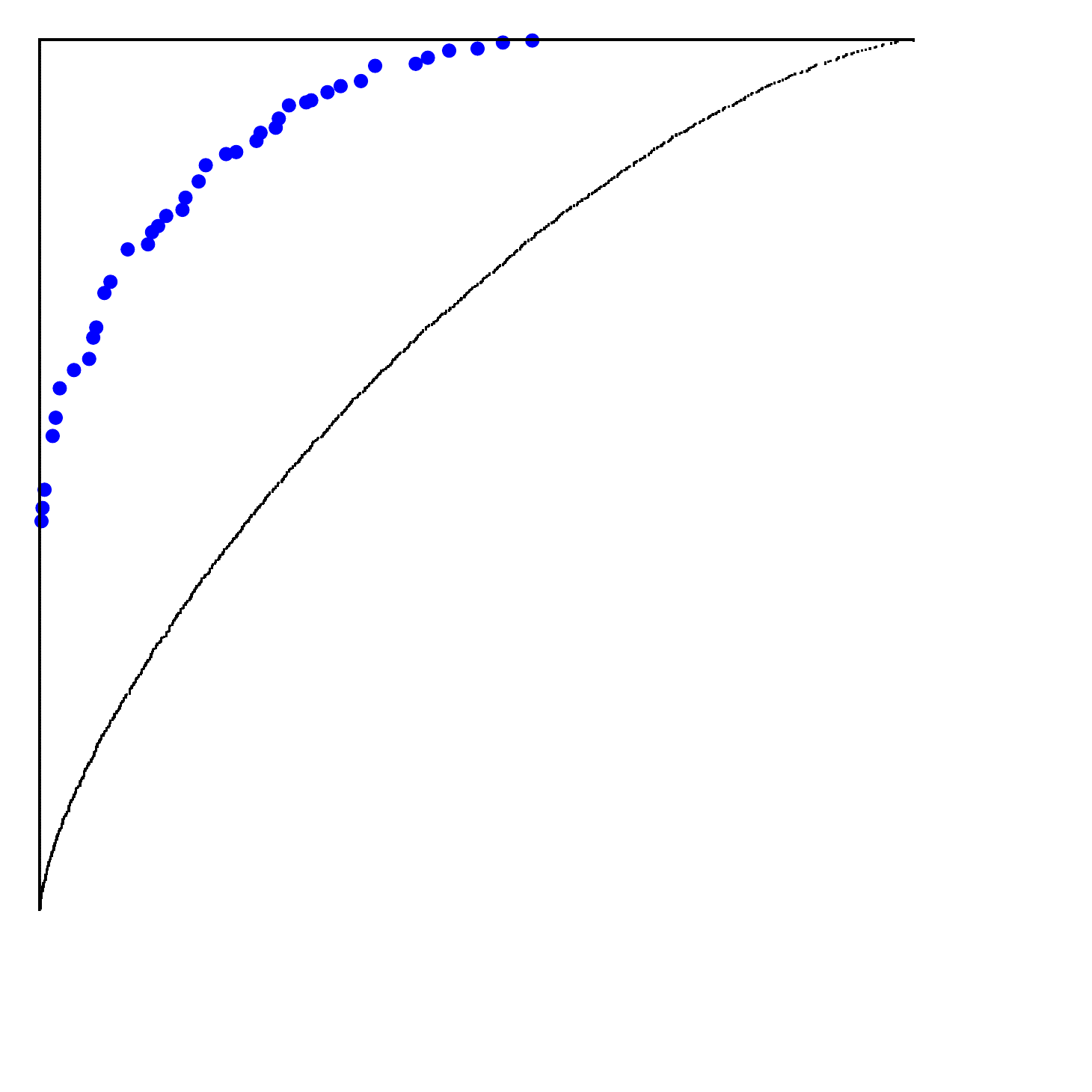}
\includegraphics[scale=0.42]{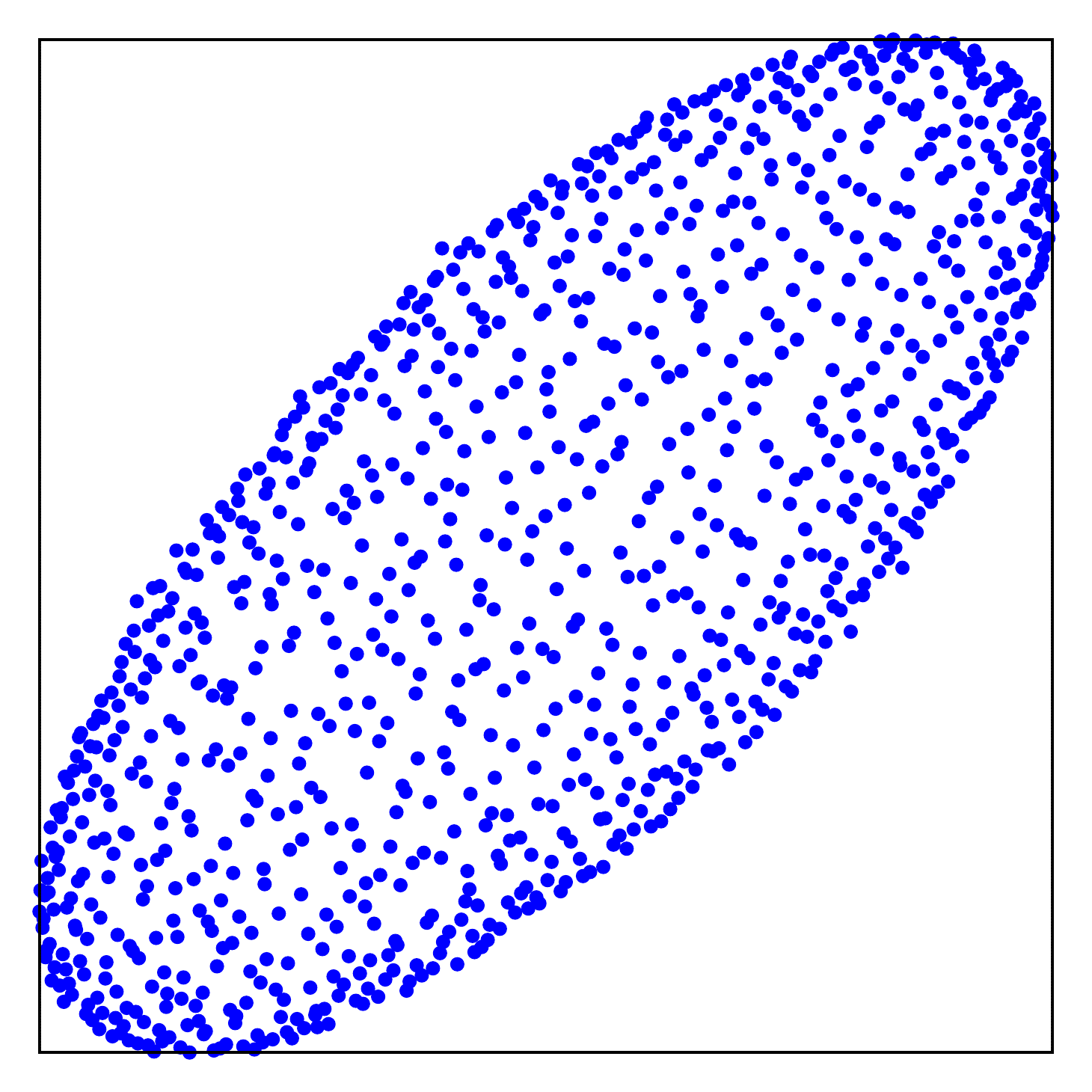}
\includegraphics[scale=0.42]{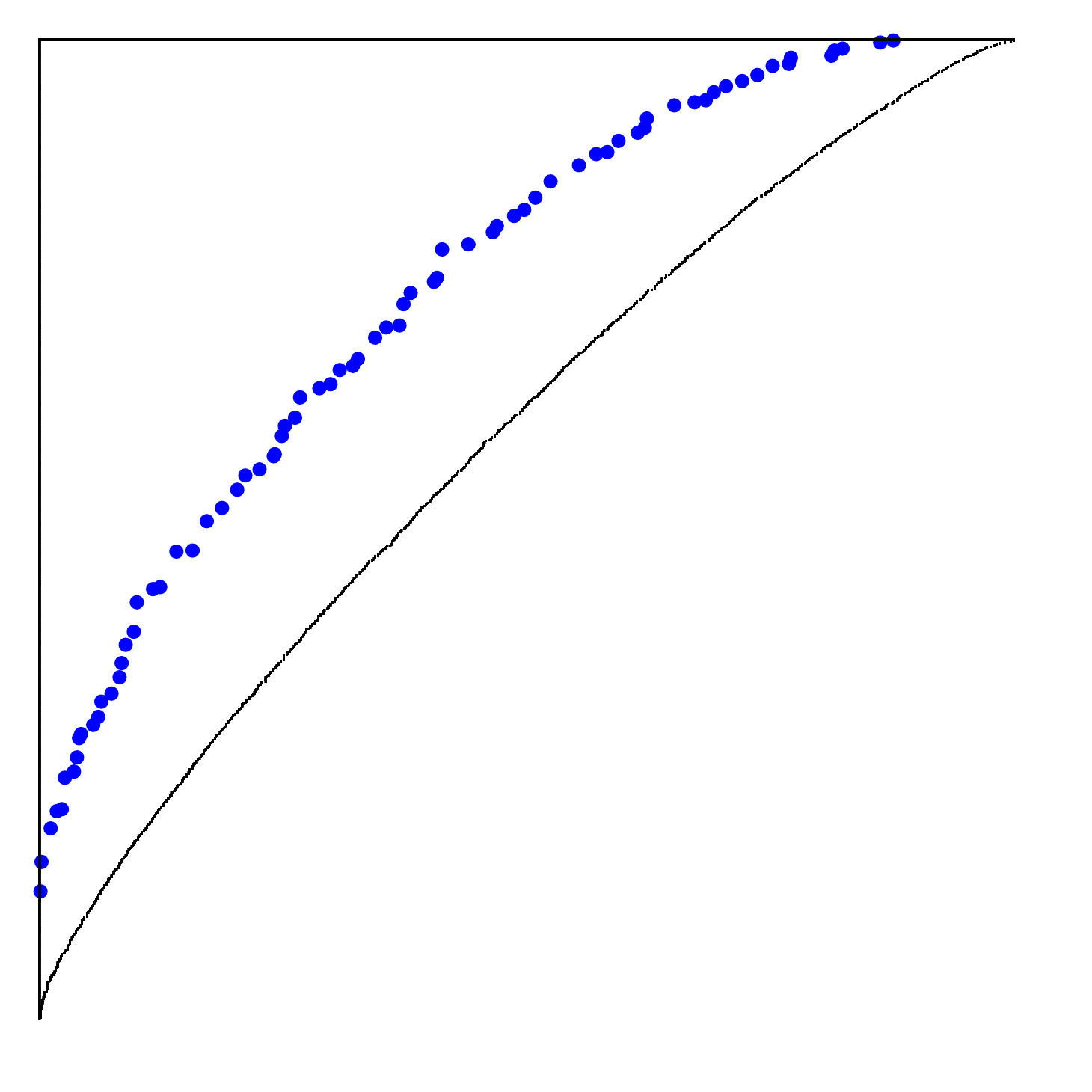}
\caption{A comparison at times $t = \frac{1}{2}$ and $\frac{3}{4}$ illustrating how the same shapes occur in both permutation matrices $M(\sigma_t)$ and frozen regions $\lambda_f(w_1\dots w_{\lfloor t \binom{n}{2} \rfloor})$, where $\sigma_t$ is the permutation defined by $w_1\dots w_{\lfloor t \binom{n}{2} \rfloor}$ for a random 1000-element sorting network $w$.}
\label[fig]{comparison}
\end{figure}

\subsection{Pattern avoidance}

\cref{frozendiagram} also connects our work with the study of pattern-avoiding permutations. The permutation $\sigma \in \S_n$ \emph{contains} the pattern $p = p_1 \dots p_k \in \mathbb{N}^*$ if there exist indices $1 \leq i_1 < i_2 \dots < i_k \leq n$ such that $\sigma(i) < \sigma(j)$ if and only if $p_i < p_j$ for all $i < j$, $i, j \in \{i_1, \dots, i_k\}$. If $\sigma$ does not contain $p$, it is called \emph{$p$-avoiding}. The set of 132-avoiding permutations of $[n]$, $\S_n(132)$, is of particular interest here. The reason is an observation of Fulton.

\begin{lem}[{\cite[Proposition 9.19]{Fu92}}]\label[lem]{fulton132}
	Let $\sigma \in \S_n$. Then $\sigma$ is 132-avoiding if and only if $D(\sigma) = D_{(1,1)}(\sigma)$.
\end{lem}

Since the length of a reduced word of $\sigma \in \S_n$ is exactly the number of inversions in $\sigma$, that is $\mathrm{inv}(\sigma)$, \cref{fulton132} suggests we also need the following well-known fact: If $\sigma \in \S_n$, then $|D(\sigma)| = \mathrm{inv}(\sigma)$. Note that by \cref{fulton132}, this can also be stated as $\lambda(\sigma) \vdash \mathrm{inv}(\sigma)$ for $\sigma \in \S_n(132)$, meaning that $\lambda(\sigma)$ is a partition of $\mathrm{inv}(\sigma)$. We then obtain the characterization below.
\begin{cor}\label[cor]{frozenp}
	Let $w \in \mathcal{R}(\sigma)$. The insertion tableau $P(w)$ is frozen if and only if $\sigma$ is 132-avoiding. 
\end{cor}

Somewhat related, Tenner showed in \cite[Theorem 5.15]{Te17} that the set of 132-avoiding permutations of any length with $k$ inversions is in bijection with partitions of $k$. The proof is by constructing a bijection by filling the Young diagram of $\lambda \vdash k$ in such a way that the result is a frozen tableau (it is called antidiagonal filling in the paper). Then the reading words of these tableaux are shown to be reduced and moreover to yield the 132-avoiding permutations. This would then imply the ``only if''-direction of \cref{frozenp} by \cref{cox}.

The corollary below is mostly a reproof of consequences of results by Stanley \cite[Theorem 4.1]{St84}, and Edelman and Greene \cite[Theorem 8.1]{EG87}. We have added the observation that each shape $\lambda \subset \mathrm{sc}_n$ appears for exactly one $\sigma \in \S_n(132)$ (and the consequent second bijection), which also follows from their works by properties of 132-avoiding permutations but is not discussed.
\begin{cor}\label[cor]{132av}
	If $\sigma$ is 132-avoiding, then $P(w)$ is frozen and has the same shape $\lambda(\sigma)$ for all $w \in \mathcal{R}(\sigma)$. Furthermore, each shape $\lambda \subset \mathrm{sc}_n$ appears for exactly one $\sigma \in \S_n(132)$. Hence, $\mathrm{EG}(w): w \mapsto Q(w)$ defines a bijection \[\mathcal{R}(\sigma) \rightarrow \mathrm{SYT}(\lambda(\sigma)),\] and a bijection \[\bigcup\limits_{\sigma \in \S_n(132)}\mathcal{R}(\sigma) \rightarrow \bigcup\limits_{\lambda \subset \mathrm{sc}_n} \mathrm{SYT}(\lambda).\]
\end{cor}

\begin{cor}\label[cor]{132f}
	Let $f^{\lambda} = |\mathrm{SYT}(\lambda)|$. Then
	\[\left\vert{\bigcup\limits_{\sigma \in \S_n(p)}\mathcal{R}(\sigma)}\right\vert = \sum\limits_{\lambda \subset \mathrm{sc}_n} f^{\lambda},\] where $p \in \{132, 213\}.$
\end{cor}
This is implied by \cref{132av} and the discussion after~\cref{complement}. However, we have not been able to simplify the sum on the right-hand side.

\subsection{132-avoiding sorting networks}
Having in mind that the insertion tableau $P(w)$ becomes frozen for any reduced word $w$ of the reverse permutation, it could be interesting to restrict to \emph{132-avoiding sorting networks}, that is, those reduced words $w = w_1 \dots w_{\binom{n}{2}} \in \mathcal{R}(n)$ such that for any $1 \leq i \leq \binom{n}{2}$ the permutation $s_{w_1} \cdots s_{w_i}$ is 132-avoiding, or, equivalently, $P(w_1 \dots w_i)$ is frozen. This corresponds to considering the maximum length chains in the weak Bruhat order on $\S_n$ restricted to 132-avoiding permutations. Bj\"orner and Wachs showed in \cite{BW97} that the restriction yields a sublattice isomorphic to the \emph{Tamari lattice} $\mathcal{T}_n$.

Using results from the next section, we can characterize 132-avoiding sorting networks in terms of \emph{shifted standard Young tableaux}, which was first proved by Fishel and Nelson~\cite[Theorem 4.6]{FN14}. These are standard Young tableaux for which each row $i$ can be shifted $(i-1)$ steps to the right without breaking the rule that the columns are increasing downwards. For example, \[\ytableausetup{centertableaux} \begin{ytableau}
\scriptstyle 1 & \scriptstyle 2 & \scriptstyle 4\\
\scriptstyle 3 & \scriptstyle 5 \\
\scriptstyle 6 
\end{ytableau} \longrightarrow\ \begin{ytableau}
\scriptstyle 1 & \scriptstyle 2 & \scriptstyle 4\\
\none & \scriptstyle 3 & \scriptstyle 5 \\
\none & \none & \scriptstyle 6 
\end{ytableau}.\]
\begin{prop}{\cite[Theorem 4.6]{FN14}}\label[prop]{132sort}
	Let $w = w_1 \dots w_{\binom{n}{2}}$ be a sorting network.\\ It is 132-avoiding if and only if $Q_{i, j} > Q_{i-1, j+1}$ for all $(i, j), (i-1, j + 1) \in Q$, or in other words, $Q$ is a shifted standard Young tableau of the shape $\mathrm{sc}_n$, where $Q = \mathrm{EG}(w)$.\\	
	It is 213-avoiding if and only if $Q_{i, j} < Q_{i-1, j+1}$ for all $(i, j), (i-1, j + 1) \in Q$ where $Q = \mathrm{EG}(w)$.
\end{prop}
\begin{proof} Suppose $Q_{i, j} > Q_{i-1, j+1}$ for all $(i, j), (i-1, j + 1) \in Q$. We shall see in \cref{zeroheight} that then $w = c_1 c_2 \dots c_{\binom{n}{2}}$ where $c_i$ is the number of the column of $1 \leq i \leq \binom{n}{2}$ in $Q$. This implies that $P^{(k)}$ is frozen for all $1 \leq k \leq \binom{n}{2}$, since each letter $c_i$ is inserted in column number $c_i$ on the first row. For example, consider $w = 121321$. Its insertion forms the sequence \begin{align*} \ytableausetup{centertableaux}
\begin{ytableau}
\scriptstyle 1 \\
\none \\
\none
\end{ytableau}\ &\longrightarrow\ \begin{ytableau}
\scriptstyle 1 & \scriptstyle 2 \\
\none \\
\none
\end{ytableau}\ \longrightarrow\ \begin{ytableau}
\scriptstyle 1 & \scriptstyle 2\\
\scriptstyle 2 \\
\none
\end{ytableau}\ \longrightarrow\ \begin{ytableau}
\scriptstyle 1 & \scriptstyle 2 & \scriptstyle 3\\
\scriptstyle 2 \\
\none
\end{ytableau}\\ &\longrightarrow\ \begin{ytableau}
\scriptstyle 1 & \scriptstyle 2 & \scriptstyle 3\\
\scriptstyle 2 & \scriptstyle 3\\
\none
\end{ytableau}\ \longrightarrow\ \begin{ytableau}
\scriptstyle 1 & \scriptstyle 2 & \scriptstyle 3\\
\scriptstyle 2 & \scriptstyle 3 \\
\scriptstyle 3
\end{ytableau}.\end{align*} For the other direction, assume $P^{(k)}$ is frozen for all $1 \leq k \leq \binom{n}{2}$ and suppose $w$ is not of the form $c_1 c_2 \dots c_{\binom{n}{2}}$. Then some letter $w_i$ is inserted in column $c_j > c_i$ on the first row. The letter $w_i$ bumps $c_j$. Otherwise the insertion tableau was not frozen. This means $c_j + 1$ is inserted in the second row. Either it is the largest on the row or bumps a $c_j + 1$ since the insertion tableau has to be frozen. Using this argument inductively, we see that at no point in the insertion can a letter be inserted into a column other than $c_j$. This is a contradiction. Hence $w = c_1 c_2 \dots c_{\binom{n}{2}}$, but then by \cref{zeroheight}, $Q_{i, j} > Q_{i-1, j+1}$ for all $(i, j), (i-1, j + 1) \in Q$. 

The second statement follows from the first by symmetries. \end{proof}
This subclass of sorting networks has also been studied by Schilling, Thi\'ery, White and Williams in \cite{STWW17}. Note in particular the observation that 132-avoiding sorting networks form a commutation class, that is, each 132-avoiding sorting network is reachable from another by a sequence of commutations: $s_i s_j \mapsto s_j s_i$ if $|i-j| > 1$. They also observed that by \cite[Lemma 2.2]{STWW17} $n$-element 132-avoiding sorting networks are in bijection with reduced words of the signed permutation ${-(n-1)}\ {-(n-2)}\ \dots\ {-1}$ by $s_i \mapsto s_{i-1}$. 

Another characterization of 132-avoiding sorting networks is in terms of lattice words (also called lattice permutations or Yamanouchi words). A \emph{lattice word} of type $\lambda = (\lambda_1, \dots, \lambda_m)$ is a word $w = w_1 \dots w_m$ in which for each $2 \leq i+1 \leq m$ there is at least one $i$ before it, and $i$ occurs $\lambda_i$ times in $w$.
\begin{prop}\label[prop]{132lattice}
	Let $w = w_1 \dots w_{\binom{n}{2}}$ be a sorting network and let $\bar{w} = \bar{w}_1 \dots \bar{w}_{\binom{n}{2}}$, where $\bar{w}_i = n - w_i$ for $1 \leq i \leq \binom{n}{2}$. Then $w$ is 132-avoiding if and only if $w$ (or equivalently, $w^{\text{rev}}$) is a lattice word of type $\mathrm{sc}_n$.
	It is 213-avoiding if and only if $\bar{w}$ (or equivalently, $\bar{w}^{\text{rev}}$) is a lattice word of type $\mathrm{sc}_n$.
\end{prop}
\begin{proof} The proof borrows from the proof of \cref{132sort}. Suppose $w$ is a 132-avoiding sorting network. Then, by \cref{132sort} and \cref{zeroheight}, $w = c_1 c_2 \dots c_{\binom{n}{2}}$ where $c_i$ is the column of $1 \leq i \leq \binom{n}{2}$ in $Q(w)$. This implies that $w$ is a lattice word of type $\mathrm{sc}_n$. For the other direction, note that if $w$ is a lattice word of type $\mathrm{sc}_n$, then the $P^{(k)}$ obtained in computing $\mathrm{EG}(w)$ are frozen for all $1 \leq k \leq \binom{n}{2}$. By \cref{frozenp}, $w$ is 132-avoiding. 

The second statement follows from the first. \end{proof}

Fishel and Nelson proved the ``$\Rightarrow$''-direction of \cref{132lattice} in \cite[Corollary 4.5]{FN14}. Note that if $w = w_1 \dots w_k$ is a 132-avoiding sorting network, $w^{\text{rev}} = w_k \dots w_1$ is a 132-avoiding sorting network as well, since $Q(w^{\text{rev}})$ can be obtained by shifting $Q(w)$, reflecting the result anti-diagonally, complementing the entries: $m \mapsto \binom{n}{2} - m + 1$, and (un)shifting back.

We should emphasize that 132-avoiding and 312-avoiding sorting networks coincide.
\begin{prop}
	A sorting network is 132-avoiding if and only if it is 312-avoiding. Similarly, a sorting network is 213-avoiding if and only if it is 231-avoiding.
\end{prop}
\begin{proof} Suppose that a 132-avoiding sorting network is not 312-avoiding. This means that an intermediate permutation contains the pattern 312. It must have been created by swapping the 1 and the 3. Hence, a previous intermediate permutation contains the pattern 132, a contradiction. If a 312-avoiding sorting network is not 132-avoiding, an intermediate permutation contains the pattern 132. The 1 and the 3 are swapped in a later intermediate permutation, which leads to a contradiction. A similar argument applies to 213-avoiding and 231-avoiding sorting networks. \end{proof}

The following enumerative result was first obtained by Fishel and Nelson \cite[Corollary 3.4]{FN14} who enumerated the maximum length chains in $\mathcal{T}_n$ using a different set of methods. However, it is also a reformulation of \cref{heightenum} by \cref{132sort}. It can be reinterpreted in terms of several other combinatorial objects, for example Gelfand-Tsetlin patterns (see the entry A003121 in the OEIS \cite{OEIS}).

\begin{cor}[{\cite[Corollary 3.4]{FN14}}]\label[prop]{132enum}
	The number of 132-avoiding sorting networks of length $\binom{n}{2}$ is \[{\binom{n}{2}}!\frac{1!2!\dots(n-2)!}{1!3!\dots(2n-3)!}.\] The same holds for 213-avoiding sorting networks.
\end{cor}

The study of 132-avoiding sorting networks is continued in~\cite{LPS18}.

\subsection{Vexillary permutations}
The proof method of \cite[Theorem 8.1, part 2]{EG87} would also lead to a proof of \cref{frozenp}. Moreover, it in fact allows us to prove something stronger. A permutation is said to be \emph{vexillary} if it is 2143-avoiding. 
For $(i, j) \in D(\sigma)$, let $r(i, j)$ be the \emph{rank} of $(i, j)$, the number of 1s north-west of $(i, j)$ in $M(\sigma)$. The result below provides a direct bijection between insertion tableaux of reduced words of vexillary permutations and their diagrams. 
	\begin{thm}\label[thm]{vex}
		Let $w \in \mathcal{R}(\sigma)$. If $\sigma$ is vexillary, then the cell $(i, j)$ of $P(w)$ contains the entry $(i + j - 1) + k, k \geq 0$, if and only if $(i + k, j + k)$ is in $D(\sigma)$, where $k = r(i+k, j+k)$. Furthermore, if the set consisting of the cells $(i + k, j + k)$ for entries $(i + j - 1) + k, k \geq 0$, in cells $(i, j)$ in $P(w)$ is the diagram of a vexillary permutation, then $\sigma$ is vexillary.
	\end{thm}
\begin{proof} We prove this by using the construction in the proof of \cite[Theorem 8.1, part 2]{EG87}. The idea is as follows. For a permutation $\sigma$, create a row (with possible empty spaces) of cells, the columns $x$ containing the positions $x$ such that $\sigma(y) < \sigma(x)$ for some $y > x$. Next, for each $x$ in the row, add $x + 1, \dots, x + r_x(\sigma) - 1$, where $r_x = \vert\{y : y > x, \sigma(y) < \sigma(x)\}\vert$, below $x$ in the same column. Note that $r_x$ is the $x$:th component of the (Lehmer) code of $\sigma$, that is, the number of inversions whose smaller component is $x$. Denote this configuration of cells by $T_0(\sigma)$. Finally, left-justify the rows and call the resulting increasing tableau $T(\sigma)$. It follows from \cite[Theorem 8.1, part 1]{EG87} and the proof of \cite[Theorem 8.1, part 2]{EG87} that for $\sigma$ vexillary, $T(\sigma) = P(w)$ for all $w \in \mathcal{R}(\sigma)$. As an example, $\sigma = 813975246$ is considered in \cref{constr_vex}.
\begin{figure}[htbp!]
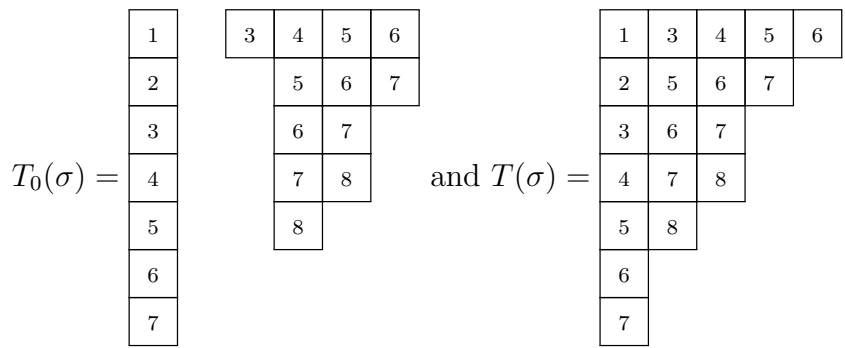

\[T_0(\sigma) = \begin{ytableau}
	\scriptstyle 1 & \none & \scriptstyle 3 & \scriptstyle 4 & \scriptstyle 5 & \scriptstyle 6\\
	\scriptstyle 2 & \none & \none & \scriptstyle 5 & \scriptstyle 6 & \scriptstyle 7\\
	\scriptstyle 3 & \none &\none & \scriptstyle 6 & \scriptstyle 7\\
	\scriptstyle 4 & \none &\none & \scriptstyle 7 & \scriptstyle 8\\
	\scriptstyle 5 & \none &\none & \scriptstyle 8\\
	\scriptstyle 6\\
	\scriptstyle 7\\
	\end{ytableau}\ \text{and}\ T(\sigma) = \begin{ytableau}
	\scriptstyle 1 & \scriptstyle 3 & \scriptstyle 4 & \scriptstyle 5 & \scriptstyle 6\\
	\scriptstyle 2 & \scriptstyle 5 & \scriptstyle 6 & \scriptstyle 7\\
	\scriptstyle 3 & \scriptstyle 6 & \scriptstyle 7\\
	\scriptstyle 4 & \scriptstyle 7 & \scriptstyle 8\\
	\scriptstyle 5 & \scriptstyle 8\\
	\scriptstyle 6\\
	\scriptstyle 7\\
	\end{ytableau}\]
\caption{The construction used in the proof of \cref{vex} for $\sigma = 813975246$. Compare with \cref{diag_vex}.}
\label[fig]{constr_vex}
\end{figure}
\begin{figure}[htbp!]
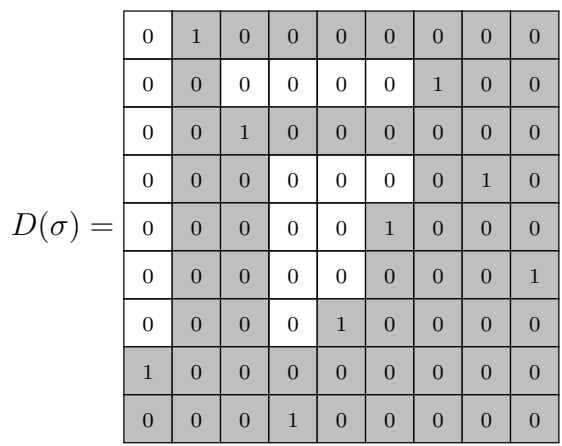

\[D(\sigma) = 
\begin{ytableau}
		\scriptstyle 0 & *(lightgray)\scriptstyle 1 & *(lightgray)\scriptstyle 0 & *(lightgray)\scriptstyle 0 & *(lightgray)\scriptstyle 0 & *(lightgray)\scriptstyle 0 & *(lightgray)\scriptstyle 0 & *(lightgray)\scriptstyle 0 & *(lightgray)\scriptstyle 0\\
		\scriptstyle 0 & *(lightgray)\scriptstyle 0 & \scriptstyle 0 & \scriptstyle 0 & \scriptstyle 0 & \scriptstyle 0 & *(lightgray)\scriptstyle 1 & *(lightgray)\scriptstyle 0 & *(lightgray)\scriptstyle 0 \\
		\scriptstyle 0 & *(lightgray) \scriptstyle 0 & *(lightgray)\scriptstyle 1 & *(lightgray)\scriptstyle 0 & *(lightgray)\scriptstyle 0 & *(lightgray)\scriptstyle 0 & *(lightgray)\scriptstyle 0 & *(lightgray)\scriptstyle 0 & *(lightgray)\scriptstyle 0 \\
		\scriptstyle 0 & *(lightgray) \scriptstyle 0 & *(lightgray)\scriptstyle 0 & \scriptstyle 0 & \scriptstyle 0 & \scriptstyle 0 & *(lightgray)\scriptstyle 0 & *(lightgray)\scriptstyle 1 & *(lightgray)\scriptstyle 0\\
		\scriptstyle 0 & *(lightgray) \scriptstyle 0 & *(lightgray)\scriptstyle 0 & \scriptstyle 0 & \scriptstyle 0 & *(lightgray)\scriptstyle 1 & *(lightgray)\scriptstyle 0 & *(lightgray)\scriptstyle 0 & *(lightgray)\scriptstyle 0\\
		\scriptstyle 0 & *(lightgray) \scriptstyle 0 & *(lightgray)\scriptstyle 0 & \scriptstyle 0 & \scriptstyle 0 & *(lightgray)\scriptstyle 0 & *(lightgray)\scriptstyle 0 & *(lightgray)\scriptstyle 0 & *(lightgray)\scriptstyle 1\\
		\scriptstyle 0 & *(lightgray) \scriptstyle 0 & *(lightgray)\scriptstyle 0 & \scriptstyle 0 & *(lightgray)\scriptstyle 1 & *(lightgray)\scriptstyle 0 & *(lightgray)\scriptstyle 0 & *(lightgray)\scriptstyle 0 & *(lightgray)\scriptstyle 0\\
		*(lightgray)\scriptstyle 1 & *(lightgray)\scriptstyle 0 & *(lightgray)\scriptstyle 0 & *(lightgray)\scriptstyle 0 & *(lightgray)\scriptstyle 0 & *(lightgray)\scriptstyle 0 & *(lightgray)\scriptstyle 0 & *(lightgray)\scriptstyle 0 & *(lightgray)\scriptstyle 0\\
		*(lightgray)\scriptstyle 0 & *(lightgray)\scriptstyle 0 & *(lightgray)\scriptstyle 0 & *(lightgray)\scriptstyle 1 & *(lightgray)\scriptstyle 0 & *(lightgray)\scriptstyle 0 & *(lightgray)\scriptstyle 0 & *(lightgray)\scriptstyle 0 & *(lightgray)\scriptstyle 0\\
	\end{ytableau}
\]
\caption{The diagram of $\sigma = 813975246$. Compare with \cref{constr_vex}.}
\label[fig]{diag_vex}
\end{figure}

Consider the connected component $D_{(i+k, j+k)}$ in the diagram of a vexillary permutation $\sigma$ having its north-west corner in $(i+k, j+k)$, where $k$ is the number of 1s north-west of $(i+k, j+k)$. Note that $k$ is well-defined. Assume that $D_{(i+k, j+k)}$ has column lengths $c_0, \dots, c_{l}$. 

We first show that for $0 \leq m \leq l$, column $j + k + m$ of $T_0(\sigma)$ has at least $c_m$ entries weakly south of row $i$. These entries are then by construction $(i + j - 1) + k + m, \dots, (i + j - 1) + k + m + c_m - 1$ as required in $P(w)$. It is clear that there are exactly $c_m$ 1s east of column $j+k + m$, north of row $(i + k) + c_m$ and weakly south of $i+k$, whereas the 1-entry of column $j+k + m$ must lie weakly south of $(i + k) + c_m$. See \cref{illustration_vex}. Furthermore, there are exactly $k$ 1s north-west of $(i + k, j + k)$ in the permutation matrix. Hence $i-1$ 1s are strictly north-east of the component with north-west corner in $(i+k, j+k)$. Hence column $j+k +m$ of $T_0(\sigma)$ contains at least the entries $j+k+m, \dots (j+k) + (i-2) + m, (i + j - 1) + k + m, \dots, (i + j - 1) + k + m + c_m - 1$. This proves the claim.
\begin{figure}[htbp!]
\[
\setlength{\arrayrulewidth}{1.7pt}
\begin{array}{@{\hspace{0pt}}c@{\hspace{0pt}}c}
\begin{array}[t]{@{\hspace{0pt}}r}\vphantom{\begin{sideways} $\scriptscriptstyle j+k+l$\end{sideways}} \vspace{\arrayrulewidth} \\ \begin{ytableau} \none \end{ytableau}\\ \begin{ytableau} \none \end{ytableau} \\ \begin{ytableau} \none \end{ytableau} \\ \begin{ytableau} \none \end{ytableau} \\ 
\begin{ytableau} \none \end{ytableau} \vspace{\arrayrulewidth} \\ \begin{ytableau} \none \end{ytableau} \vspace{\arrayrulewidth} \\ \begin{minipage}[c][1.5em]{4ex}\flushright $\scriptscriptstyle i+k$\end{minipage} \\ \begin{ytableau} \none \end{ytableau} \\ \begin{ytableau} \none \end{ytableau} \\ \begin{ytableau} \none \end{ytableau} \\ \begin{minipage}[c][1.5em]{8ex}\flushright $\scriptscriptstyle i+k+c_0$\end{minipage} \\ \begin{ytableau} \none \end{ytableau} \\ \begin{ytableau} \none \end{ytableau}\end{array} & \begin{array}[t]{|@{\hspace{0pt}}c@{\hspace{0pt}}|@{\hspace{0pt}}c@{\hspace{0pt}}|@{\hspace{0pt}}c@{\hspace{0pt}}|@{\hspace{0pt}}c@{\hspace{0pt}}|@{\hspace{0pt}}c@{\hspace{0pt}}}
\multicolumn{1}{c}{} & \multicolumn{1}{c}{} & \multicolumn{1}{l}{\begin{sideways}$\scriptscriptstyle j+k$\end{sideways} \quad\quad\quad\quad\quad\quad\begin{sideways}$\scriptscriptstyle j+k+l$\end{sideways}} & \multicolumn{1}{c}{}\\
\hline
\begin{minipage}[c]{20mm}\centering $k-1$ \\ 1-entries \\ \end{minipage} & \begin{ytableau} 0 \\ \vertdots \\ \vertdots \\ \vertdots \\ 0 \end{ytableau} & \textbf{\Large 0} & \begin{minipage}[c]{18mm}\centering $i-1$ \\ 1-entries \\ \end{minipage}\\
\hline
\begin{ytableau} 0 & \dots & \dots & \dots & \dots & \dots & 0 \end{ytableau} & \begin{ytableau} *(lightgray)1 \end{ytableau} & \begin{ytableau} *(lightgray)0 & *(lightgray)\dots & *(lightgray)0 & *(lightgray)0 & *(lightgray)0 \end{ytableau} & \begin{ytableau} *(lightgray)0 & *(lightgray)\dots & *(lightgray)0 \end{ytableau} \\ 
\hline
\textbf{\Large 0} & \begin{ytableau} *(lightgray)0 \\ *(lightgray)0 \\ *(lightgray)0\\ *(lightgray)0 \end{ytableau} & \begin{ytableau}
*(cyan)0 & *(cyan)\dots & *(cyan)0 & *(cyan)0 & *(cyan)0\\
*(cyan)0 & *(cyan)\dots & *(cyan)0 & *(cyan)0 & *(lightgray)1\\
*(cyan)0 & *(cyan)\dots & *(cyan)0 & *(cyan)0 & *(lightgray)0\\
*(cyan)0 & *(cyan)\dots & *(cyan)0 & *(lightgray)1 & *(lightgray)0\\
\end{ytableau} & \\ \hline
\begin{minipage}[c]{20mm}\centering $j-1$ \\ 1-entries \\ \end{minipage} & \begin{ytableau} *(lightgray)0 \\ *(lightgray)\vertdots \\ *(lightgray)0 \end{ytableau} & & \\
\hline
\end{array}
\end{array}
\]
\caption{An example illustrating the proof of \cref{vex}. The component $D_{(i+k, j+k)}(\sigma)$ is in cyan and the entries shaded are not in $D(\sigma)$.}
\label[fig]{illustration_vex}
\end{figure}

Next we show that the part of $T_0(\sigma)$ corresponding to $D_{(i+k, j+k)}$ is shifted to the left by $k$ steps in $T(\sigma)$. No 1s appear west of the component. Hence all columns of $T_0(\sigma)$ left of $j + k$ have either length shorter than $i$ (the rows of the 1s north-west of $D_{(i+k, j+k)}$ have to be increasing since otherwise a pair of 1s would form the pattern 2143 together with the 1s in column $j+k$ and row $i+k$) or longer than $i + c_0 - 1$. Since $r(i+k, j+k) = k$, exactly $k$ columns are shorter. This proves that $(x, y) \in T(\sigma)$ contains $(x+y-1) + k$ for all $(x+k, y+k) \in D_{(i+k, j+k)}$.

We have now proved an alternative description of the map $T: \sigma \mapsto T(\sigma)$ in terms of the diagram of $\sigma$ if $\sigma$ is vexillary. The $``\Rightarrow"$-direction then follows from the bijectivity of $T$. Note that $T^{-1}$ can be described in terms of the diagram as sending $(i, j)$ with entry $(i+j-1) + k$ to $(i+k, j+k)$ (where $k = r(i+k, j+k)$).

The diagram description of $T^{-1}$ does not apply to general permutations. However, we can regardless define the same map for the insertion tableau $P(w)$ of any reduced word $w$. Let us call it $\delta: P(w) \mapsto \delta(P(w)) \subset \mathbb{N}^2$. The last statement claims that the pre-image of diagrams of vexillary permutations under $\delta$ is precisely the insertion tableaux of reduced words of vexillary permutations. First, note that by \cref{cox}, if $\sigma_1 \neq \sigma_2$ and $w_1 \in \mathcal{R}(\sigma_1), w_2 \in \mathcal{R}(\sigma_2)$, then $P(w_1) \neq P(w_2)$. Hence, it suffices to show that $\delta$ is injective. If the cell $(i_1, j_1)$ of $P(w)$ contains $(i_1 + j_1 - 1) + k_1$, then, since $P(w)$ is row and column strict, a cell $(i_2, j_2)$ of $P(w)$ with $i_2 \geq i_1$, $j_2 \geq j_1$ has to contain some $(i_2 + j_2 - 1) + k_2$  with $k_2 \geq k_1$, and $k_2 = k_1$ precisely when $(i_1, j_1) = (i_2, j_2)$. In other words, $(i_1, j_1)$ is sent to $(i_1 + k_1, j_1 + k_1)$, $(i_2, j_2)$ to $(i_2 + k_2, j_2 + k_2)$, and $(i_1 + k_1, j_1 + k_1) = (i_2 + k_2, j_2 + k_2)$ if and only if $(i_1, j_1) = (i_2, j_2)$, that is, different cells cannot be sent to the same cell in $\delta(P(w))$. This implies that $\delta$ is invertible and hence injective.
\end{proof} Note that one could as well use the construction in~\cite[p.~357]{BJS93} in the first paragraph of the proof. We should also remark that the entries with $k = 0$ are in the frozen region of $P(w)$. 

\section{Non-reduced words} \label{S:nonreduced}
The Edelman--Greene bijection takes as its argument a reduced word. In order to understand the insertion better, we study its interaction with non-reduced words as well. Simultaneously, we obtain \cref{zeroheight} which can be used to prove \cref{132sort} and \cref{132lattice}.

Fix a standard Young tableau $Q$ and let $\mathcal{W}_Q = \{w \in \mathbb{N}^* :  \mathrm{EG}(w) = Q, P(w) \ \mathrm{frozen}\}$. Recall that by \cref{frozenp} the reduced words in the sets $\mathcal{W}_Q$ are reduced words of 132-avoiding permutations. Note that since the tableau $Q(w) = \textrm{EG}(w)$ has $\textrm{len}(w)$ entries, all words in $\mathcal{W}_Q$ have the same length.  Also, since the Edelman--Greene correspondence is a bijection between $\mathcal{R}(\sigma)$ and $\mathrm{SYT}(\lambda(\sigma))$ for $\sigma \in \S_n(132)$, $\mathcal{W}_Q$ contains exactly one reduced word.

We define the poset $\mathcal{P}_Q = (\mathcal{W}_Q, \preceq)$ by setting $v \preceq w$ for $v, w \in \mathcal{W}_Q$ if $v_i \leq w_i$ for all $1 \leq i \leq \textrm{len}(v) = \textrm{len}(w)$. \cref{posets} contains some examples.

\begin{figure}[!h]
	\ytableausetup{smalltableaux}
	\centering
	\begin{minipage}[b]{0.32\textwidth}
	\centering
	\includegraphics[scale=0.5]{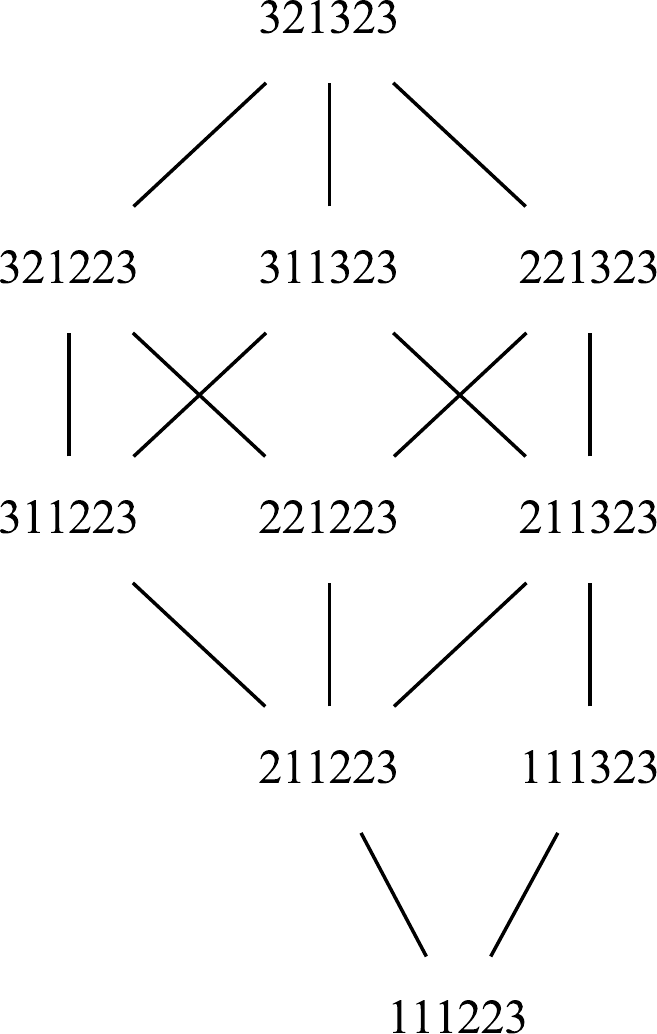}
	\\\bigskip\noindent
	\begin{ytableau}
		\scriptstyle 1 & \scriptstyle 4 & \scriptstyle 6 \\
		\scriptstyle 2 & \scriptstyle 5 \\
		\scriptstyle 3
		\end{ytableau}
	\end{minipage}
	\begin{minipage}[b]{0.2\textwidth}
	\centering
	\includegraphics[scale=0.5]{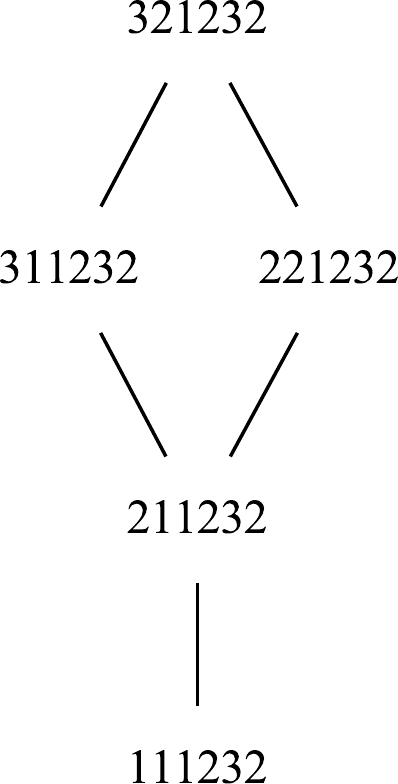}
	\\\bigskip
	\begin{ytableau}
	\scriptstyle 1 & \scriptstyle 4 & \scriptstyle 5 \\
	\scriptstyle 2 & \scriptstyle 6 \\
	\scriptstyle 3
	\end{ytableau}
	\end{minipage}
	\begin{minipage}[b]{0.2\textwidth}
	\centering
	\includegraphics[scale=0.5]{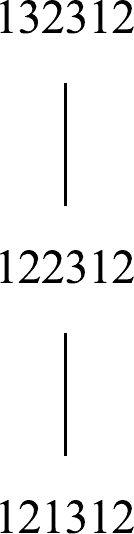}
	\\\bigskip\noindent
	\centering
	\begin{ytableau}
	\scriptstyle 1 & \scriptstyle 2 & \scriptstyle 4 \\
	\scriptstyle 3 & \scriptstyle 6 \\
	\scriptstyle 5
	\end{ytableau}
	\end{minipage}
	\caption{Some examples of the 16 posets $\mathcal{P}_Q$ for $Q \in \mathrm{SYT}(\mathrm{sc}_4)$. The bottom elements are the column words of the respective tableaux below.}
	\label[fig]{posets}
\end{figure}

\subsection{Properties of \texorpdfstring{$\mathcal{P}_Q$}{PQ}}
First, we extend a result of Edelman and Greene. The \emph{descents} of a standard Young tableau $T$ are entries $k$ such that if $T_{i, j} = k$, then $T_{i', j'} = k + 1$ for $i' > i$, in other words $k + 1$ is strictly south of $k$. Let $\mathrm{Des}(T) = \{k : k \text{ is a descent of } T\}$ be the set of descents of $T$. Correspondingly, for $w \in \mathbb{N}^*$, let $\mathrm{Des}(w) = \{1 \leq i \leq \mathrm{len}(w) - 1 : w_i \geq w_{i+1}\}$. The elements of $\mathrm{Des}(w)$ are called the \emph{weak descents} of $w$.

At times, in particular in the following proof, we refer to bumping paths. Consider constructing EG$(w)$ for an arbitrary word $w = w_1 \cdots w_m$. When $w_k$ is inserted, some entries $P^{(k-1)}_{i, j}$ of $P^{(k-1)}$ may be bumped. We let the \emph{bumping path} $p^w_{w_k}$ of $w_k$ be the set of the corresponding cells $(i, j) \in P^{(k-1)}$.

\begin{prop}\label[prop]{descents}
	For all $w \in \mathcal{P}_Q$, $\mathrm{Des}(w) = \mathrm{Des}(Q)$.
\end{prop}
\begin{proof} The proof is based on an extension of Lemma 6.28 in \cite{EG87} (which is analogous to a property of the RSK correspondence). Let $w = w_1 \dots w_i \dots w_m$. First suppose $i \in \mathrm(w)$. In running the Edelman--Greene insertion, when $x_i$ and $y_i$, $x_i \geq y_i,$ are inserted consecutively on row $i$, $x_i$ either becomes the last entry of that row or bumps some $x' \geq x_i$, and $y_i$ bumps some $y'$, $x_i \geq y' \geq y_i$. Hence, $x_{i+1} \geq x_i + 1$ and $y_{i+1} \leq x_i + 1$. Using this argument inductively shows that $p^w_{w_{i+1}}$ is weakly to the left of $p^w_{w_{i}}$. Thus $i + 1$ ends up on a lower row than $i$ in $Q = \mathrm{EG}(w)$, so a weak descent of the word becomes a descent in $Q$. 

For the converse, suppose $1 \leq i \leq m-1$ is not a weak descent in $w$, which means $w_i < w_{i+1}$. Again, when $x_i$ and $y_i$, $x_i < y_i,$ are inserted consecutively on row $i$, $x_i$ either becomes the last entry of that row or bumps some $x' \geq x_i$. Next, since $y_i > x_i$, it is either inserted at the end of the row or bumps some $y' \geq y_i> x_i$. Furthermore, since the insertion tableaux always have increasing rows, $y' > x'$. Hence, $x_{i+1} < y_{i+1}$, except possibly in the case $x_i$ bumped an $x_i$, and $y' = x_i + 1$. But then necessarily $y_i = x_i + 1$, meaning that $x_{i+1} = x_i + 1$ and $y_{i+1} = x_i + 2$, so $x_{i+1} < y_{i+1}$. Repeating this argument inductively, we get that $p^w_{w_{i+1}}$ is strictly to the right of $p^w_{w_{i}}$. Hence, $i + 1$ cannot end up in a lower row than $i$, and $i$ is not a descent in $Q$. \end{proof}

Suppose $Q$ is a standard Young tableau with $m$ entries. Define $c(Q) = c_1 \dots c_i \dots c_m$, where $c_i$ is the column of $i$ in $Q$ for $1 \leq i \leq m$. Then we say that $c(Q)$ is the \emph{column word} of $Q$. See \cref{posets} for examples. Note that this term is used differently by other authors. Column words of standard Young tableaux are, by their definition, lattice words.

\begin{prop}\label[prop]{topbottom}
	For $Q \in \mathrm{SYT(sc}_n)$, $\hat{0} = c(Q)$ is the unique minimal element in $\mathcal{P}_Q$.
\end{prop}
\begin{proof}
Since $Q$ is a standard Young tableau, if $c_i$ are the columns of the entries $i$ of $Q$, $|\{i \le j : c_i = x\}| \geq |\{i \leq j : c_i = x+1\}|$ for $1 \le j \le \binom{n}{2}, 1 \leq x \leq n - 2$. Otherwise there is a row of $Q$ which is not increasing. Since $c(Q)$ has this form, each letter $x$ will end up in the $x$:th column in the Edelman--Greene insertion, $P(w)$ will be of the frozen form, and the $Q$-tableau has the entry $i$ in column $c_i$. Hence, $c(Q) \in \mathcal{P}_Q$. By the same argument, if any of the $c_i$'s is replaced by a smaller number, the shape of $P$ (and $Q$) changes. Thus $c(Q)$ is a minimal element in $\mathcal{P}_Q$. Since the columns of the insertion tableaux are always strictly increasing, the bumping paths in the Edelman--Greene insertion go down and to the left. If there is another minimal element $w$ in $\mathcal{P}_Q$, then it has to have a letter $w_i < c_i$. But then $w_i$ is inserted to a cell strictly before the $c_i$:th cell on the first row in the insertion tableau and $i$ cannot end up in the column $c_i$ as it does in $Q$. Hence $c(Q)$ is the unique minimal element in $\mathcal{P}_Q$. 
\end{proof}
We conjecture that $\mathrm{EG}^{-1}(Q)$ is maximal in $\mathcal{P}_Q$. However, in general it is not the unique maximal element. As an example, take a reduced word of the reverse permutation in $\S_6$ starting $4521343\dots$ and a non-reduced word $2431343\dots$ in the same poset $\mathcal{P}_Q$, both ending with the same subword. They are incomparable in $\mathcal{P}_Q$.

The \emph{height} $h(P)$ of a poset $P$ is the length of its longest chain. Let $[\cdot, \cdot]$ denote an interval in $\mathcal{P}_Q$ and $\ell_Q = h([c(Q), \mathrm{EG}^{-1}(Q)])$. In other words, $\ell_Q$ is the length of a maximum length chain from $c(Q)$ to $\mathrm{EG}^{-1}(Q)$. Then $\ell_Q \leq \sum_{i = 1}^{\mathrm{len}(c(Q))} (\mathrm{EG}^{-1}(Q)_i - c(Q)_i)$. However, computations suggest that we have equality for $Q \in \mathrm{SYT(sc}_n)$.
\begin{conj}\label[conj]{height}
		For $Q \in \mathrm{SYT(sc}_n)$, we conjecture that $ \mathrm{EG}^{-1}(Q)$ is a maximal element in $\mathcal{P}_Q$ and $\ell_Q = \sum_{i = 1}^{\mathrm{len}(c(Q))} (\mathrm{EG}^{-1}(Q)_i - c(Q)_i)$.
\end{conj}

Note that $\sum_{i = 1}^{\mathrm{len}(c(Q))} (\mathrm{EG}^{-1}(Q)_i - c(Q)_i)$ is the number of right slides when performing ${EG}^{-1}$ on $Q$. Hence $\ell_Q\le \binom{n}{3}$ for the shape $\mathrm{sc}_n$. Let $\eta_{n, i}$ denote the number of $Q \in \mathrm{SYT(sc}_n)$ such that $\ell_Q = i, 0 \leq i \leq \binom{n}{3}$. \cref{heights} lists some of these values.

\begin{table}[htbp!]
	\centering
\begin{tabular}[t]{c|cccccccccccc}
	$i$ & 0 & 1 & 2 & 3 & 4 & 5 & 6 & 7 & 8 & 9 & 10 \\ \hline
	$\eta_{3, i}$ & 1 & 1 \\
	$\eta_{4, i}$ & 2 & 2 & 8 & 2 & 2 \\
	$\eta_{5, i}$ & 12 & 14 & 38 & 108 & 142 & 140 & 142 & 108 & 38 & 14 & 12	
\end{tabular}
	\caption{The values of $\eta_{n, i}$ for $n = 3, 4, 5$.}
	\label{heights}
\end{table}

The tableaux $Q$ contributing to $\eta_{n, 0}$ are simple to characterize. Then $\mathcal{P}_Q$ only contains the column word $c(Q)$.
\begin{prop}\label[prop]{zeroheight}
	If $Q \in \mathrm{SYT}(\mathrm{sc}_n)$, then $\ell_Q = 0$ if and only if $Q_{i, j} > Q_{i-1, j+1}$ for all $(i, j), (i-1, j + 1) \in Q$.
\end{prop}
\begin{proof} For the if-direction, assume $Q_{i, j} > Q_{i-1, j+1}$ for all $(i, j), (i-1, j + 1) \in Q$. This means that all the anti-diagonals (sets of cells with sums of the coordinates fixed) have entries increasing downwards. Suppose that there is at least one right slide, and consider the first such occurrence. Then either $x > y$ have moved to the same anti-diagonal, into some cells $(i, j)$ and $(i-1, j+1)$, respectively, or the right slide occurs at the top of column $j+1$. Since this is the first occurrence of a right slide, no evacuation paths starting from columns $c > j+1$ can have crossed to column $j+1$. Hence both cases would imply that the evacuation paths have started from column $j+1$ more often than from column $j$, a contradiction since the anti-diagonals of $Q$ are increasing to the left. Hence, all evacuation paths are vertical and the labels only slide down, so EG$^{-1}(Q) = c(Q)$ and $\ell_Q = 0$.

For the other direction, suppose there are some $x = Q_{i, j} < Q_{i-1, j+1} = y$. If there are no right slides, then $x$, $y$, and the labels above them have to stay in the same columns. If at some point $x$ and $y$ are on the same row, then there must have been a right slide involving some element in the column of $x$. Hence $x$ and $y$ have to end up on the bottom anti-diagonal, but then an evacuation path has to start from $y$ before $x$ and some entry of the column of $x$ slides right, a contradiction. \end{proof}

Staircase standard Young tableaux satisfying the condition in \cref{zeroheight} have been enumerated in \cite{T52} (this is the shifted hook-length formula for the shifted staircase). Note that \cref{132enum} is a reformulation of the corollary below by \cref{132sort}. 
\begin{cor}\label[cor]{heightenum}
	We have \[\eta_{n, 0} = {\binom{n}{2}}!\frac{1!2!\dots(n-2)!}{1!3!\dots(2n-3)!}.\]
\end{cor}

We end with some consequences of \cref{height}.
\begin{prop}\label[prop]{consequences} Assume \cref{height} holds and $Q \in \mathrm{SYT(sc}_n)$. Then\\	
	\indent a) $\ell_{Q^t} = \binom{n}{3} - \ell_{Q}$, so the sequence $\eta_{n, i}, 0 \leq i \leq \binom{n}{3}$, is symmetric,\\
    \indent b) the Sch\"utzenberger involution $S$ satisfies $\ell_Q = \ell_{Q^S}$,\\
    \indent c) the number $\eta_{n, i}$ is even for all $n \geq 4$, $0 \leq i \leq \binom{n}{3}$,\\  
	\indent d) and $\ell_Q = \binom{n}{3}$ if and only if $Q_{i, j} < Q_{i-1, j+1}$ for all $(i, j), (i-1, j + 1) \in Q$.
\end{prop}
\begin{proof} a) By \cref{complement}, $\ell_{Q^t} = \sum_{i = 1}^{\mathrm{len}(w)} (n - 1 - w_i - c(Q)_i) = \binom{n}{2}(n-1) - \binom{n + 1}{3} - \sum_{i = 1}^{\mathrm{len}(w)} w_i = \binom{n}{3} - (\sum_{i = 1}^{\mathrm{len}(w)} w_i - \binom{n}{3}) = \binom{n}{3} - \ell_{Q}$.\\

b) Let $w \in \mathcal{R}(n)$ and $Q = Q(w)$. By \cref{reverse}, $Q(w^{rev}) = Q^S$, and by \cref{height}, we have $\ell_Q = \sum_{i = 1}^{\mathrm{len}(w)} (w_i - c(Q)_i) = \sum_{i = 1}^{\mathrm{len}(w)} w_i - \binom{n+1}{3} = \sum_{i = 1}^{\mathrm{len}(w^{rev})} w^{rev}_i - \binom{n+1}{3} = \ell_{Q^S}.$\\

c) By b), the involution $S$ satisfies $\ell_Q = \ell_{Q^S}$. Thus it suffices to prove that it is fixed-point-free for $\mathrm{SYT}(\mathrm{sc}_n)$, $n \geq 4$. This can be seen from \cref{reverse}: $Q(w^{\mathrm{rev}}) = Q(w)^S$ for $w \in \mathcal{R}(n)$, so if $S$ had a fixpoint, there would exist $w \in \mathcal{R}(n)$ such that $w = w^{\mathrm{rev}}$. We show by induction on the length of $w$ that every $w = w^{\mathrm{rev}}$ is a reduced word of the same permutation as (in other words, is Coxeter equivalent to) $i\ (i + 1)\ \dots \ (j - 1)\ j\ (j - 1)\ \dots \ (i + 1)\ i$ for some $1 \leq i < j \leq n$, in which case $w \in \mathcal{R}(n)$ only for $n \leq 3$, when $i = 1, j = n-1.$ Clearly $w$ has to be of odd length. The base case is then that $w =  (j - 1)\ j\ (j - 1)$ or $(j + 1)\ j\ (j + 1) \approx j\ (j+1)\ j$ where $\approx$ denotes Coxeter equivalence. Consider adding the letter $x$: $x\ i\ (i + 1)\ \dots \ (j - 1)\ j\ (j - 1)\ \dots \ (i + 1)\ i\ x$. We have $i < x < j$, $x = j$, $x = j + 1$, or $x = i - 1$. Using commutations in the first case gives $i\ (i + 1)\ \dots\ x\ (x - 1)\ x\ \dots\ (j - 1)\ j\ (j - 1)\ \dots\ x\ (x - 1)\ x\ \dots\ (i + 1)\ i$, which is non-reduced. If $x = j$, we get $j \ (j - 1)\ j\ (j - 1)\ j$ in the middle, which is also non-reduced. Hence either $x = i - 1$ and we are done, or $x = j + 1$, in which case we get $(j+1)\ j\ (j+1) \approx j\ (j+1)\ j$ in the middle, and are also done.\\

d) This follows from \cref{zeroheight} by transposition. \end{proof}

\subsection*{Acknowledgements}
This paper benefited greatly from experimentation with \texttt{Sage}~\cite{sage} and its combinatorics features developed by the \texttt{Sage-Combinat}
community~\cite{Sage-Combinat}. The authors thank the referees of both the FPSAC extended abstract and this full version for their valuable comments.

\bibliographystyle{plain}
\bibliography{frozen_regions}

\end{document}